\documentclass[12pt,a4paper]{article}
\usepackage{graphicx} 
\usepackage{mathrsfs,amssymb,amsfonts,amsthm,amsmath,amsbsy,fullpage,mathtools}
\usepackage{mathtools}
\usepackage{mathabx}
\usepackage{graphicx}
\usepackage{wrapfig}
\usepackage{caption}
\usepackage{subfig}
\usepackage{dsfont}
\usepackage{float}
\usepackage{upgreek}
\usepackage{listings}
\usepackage{lipsum}
\usepackage{bbm, dsfont}
\newtheorem{theorem}{Theorem}[section]
\newtheorem{proposition}[theorem]{Proposition}
\newtheorem{corollary}[theorem]{Corollary}
\newtheorem{lemma}[theorem]{Lemma}

\newtheorem{remark}[theorem]{Remark}

\usepackage{inlinelabel}

\allowdisplaybreaks

\title{Graph distance and effective resistance \\ of the four-dimensional random walk trace}
\author{
D.~Shiraishi\footnote{Graduate School of Informatics, Kyoto University, shiraishi@acs.i.kyoto-u.ac.jp.}, S.~Watanabe\footnote{Research Institute for Mathematical Sciences, Kyoto University, swatanab@kurims.kyoto-u.ac.jp.}}

\begin{document}

\maketitle

\begin{abstract}
Refining previous results, we establish a sharp asymptotic estimate on the expected graph distance between the origin and the terminal point of the trace of the first $n$ steps of the walk. A similar conclusion is drawn for the resistance metric. 
\end{abstract}


\section{Introduction}

Let $S$ be a simple random walk on $\mathbb{Z}^{d}$ started at the origin.
For $0 \le m \le n $, we denote by ${\cal G}_{m,n}$ the random graph whose vertex set is given by the sites visited by $S$ from time $m$ up to time $n$, and whose edge set is given by the edges crossed by the walk. For the precise definition, see \eqref{gmn}.
In \cite{BurLaw}, the effective resistance $R_{n}$ between the origin and $S(n)$ in ${\cal G}_{0,n}$, and the graph distance $D_{n}$ between these endpoints in ${\cal G}_{0,n}$, were studied as quantities that reflect the geometric structure of ${\cal G}_{0,n}$.

In dimensions $d \ge 5$, where self-intersections of the random walk are relatively rare, these quantities are known to grow linearly in $n$.
By contrast, in dimensions $d=2,3$, it is conjectured in \cite{BurLaw} that for each of these quantities there exists a constant $\alpha \in (0,1)$, depending not only on the dimension $d$ but also on whether one considers $R_n$ or $D_n$, such that the expectation of this quantity grows as $n^{\alpha + o(1)}$.
Determining the exact value of $\alpha$ appears to be a problem of considerable difficulty, particularly in the case $d=3$, and it seems possible that the value may never be determined.
Moreover, to the best of our knowledge, even the existence of such an exponent $\alpha$ has not yet been established.

We now turn to the case $d = 4$.
The four-dimensional case corresponds to the critical dimension for this model, and, as is often the case at criticality, determining the correct power of the logarithmic correction would be important.
Indeed, in \cite{BurLaw}, it was conjectured that there exist constants $\beta, \gamma > 0$ such that the expectations of $R_n$ and $D_n$ behave as $n (\log n)^{-\beta + o(1)}$ and $n (\log n)^{-\gamma + o(1)}$, respectively.
However, no specific conjectural values for $\beta$ and $\gamma$ were provided in \cite{BurLaw}.

Subsequently, it was proved in \cite{exact} (see also \cite{exacta}) that the expectation of $R_n$ behaves as $n (\log n)^{-\frac{1}{2} + o(1)}$.
Although this was not stated explicitly in \cite{exact}, the methods developed there can also be adapted to show, by similar arguments, that the expectation of $D_n$ behaves as $n (\log n)^{-\frac{1}{2} + o(1)}$. In particular, this indicates that $\beta = \gamma = \frac{1}{2}$.

The main result of the present paper provides the following improvement of the above result from \cite{exact}, achieved by eliminating the $o(1)$-term.

\begin{theorem}\label{main}
Let $d=4$. There exist constants $\delta >0$ and $0 < c_{\text{res}} \le c_{\text{gra}} < \infty$ such that 
\begin{align}
& E ( R_{n}  ) = c_{\text{res}} \, n (\log n)^{-\frac{1}{2}}   \Big\{ 1 + O \Big( \big( \log n \big)^{-\delta} \Big) \Big\}, \label{main1} \\
&E ( D_n ) = c_{\text{gra}} \, n (\log n)^{-\frac{1}{2}}   \Big\{ 1 + O \Big( \big( \log n \big)^{-\delta} \Big) \Big\}. \label{main2}
\end{align}
\end{theorem}

We briefly discuss the proof of Theorem \ref{main} here.
Since the proofs for $R_n$ and $D_n$ are essentially identical, we explain the argument for $D_n$.
The basic strategy is similar to that of the proof of \cite[Theorem 1.2.3]{exact}. If the logarithmic correction for the expectation of $D_{n}$ is defined by
$$\psi (n) := \frac{E (D_{n})}{n},$$
the objective is to obtain an estimate of the difference between $\psi (n)$ and $\psi (2n)$ as precisely as possible. More precisely, the main contribution of this paper is to establish the following statement: there exists $\epsilon > 0$ such that, for all $n \ge 1$,
\begin{equation}\label{GOAL1}
\psi (2n) = \psi (n) \Big\{ 1 - \frac{\log 2 }{2 \log n} + O \big( (\log n)^{-1 - \epsilon} \big) \Big\},
\end{equation}
holds. Theorem \ref{main} can then be deduced by combining the equation \eqref{GOAL1} with the fact that $\psi$ is a slowly varying function. For further details, we refer to Lemma \ref{deriv} below.

It may be worthwhile to emphasize the difference from \cite[Theorem 1.2.3]{exact}. In \cite{exact}, it was shown that for any $\eta > 0$, there exists $N_{\eta} \ge 1$ such that, for all $n \ge N_{\eta}$,
$$\psi (2n) \le \psi (n) \Big\{ 1 - \frac{\log 2 }{2 \log n} ( 1 - \eta)  \Big\}.$$
As can be seen from \eqref{GOAL1}, the above inequality obtained in \cite{exact} is not sufficient for our purposes, and a more precise control of the error term is required.

With this in mind, let $D_{n}'$ denote the graph distance between $S(n)$ and $S(2n)$ in ${\cal G}_{n,2n}$. Since the difference between $\psi(n)$ and $\psi(2n)$ is given by
$$\psi (n) - \psi (2n) = \frac{E (D_{n} + D_{n}' - D_{2n})}{2n},$$
it is necessary to analyze the expectation on the right-hand side.
Now, suppose that there exist two times $k$ and $l$ satisfying $0 \le k \le n \le l \le 2n$ such that $S(k) = S(l)$. Intuitively, as $l-k$ becomes larger, the graph ${\cal G}_{0,2n}$ admits a more substantial shortcut, and consequently $D_{n} + D_{n}' - D_{2n}$ becomes smaller. 

In view of the above observation, we write 
$$M = \max \{ l -k \ : \ 0 \le k \le n \le l \le 2n, \ S(k) = S(l) \},$$
for the maximal intersection length and estimate the expectation of $D_{n} + D_{n}' - D_{2n}$ by considering cases according to the value of $M$. (In fact, rather than working directly with $M$, we analyze a closely related quantity $L$, defined in \eqref{L} below.) Roughly speaking, conditioned on $M = m$, the quantity $D_{n} + D_{n}' - D_{2n}$ can, with high conditional  probability, be approximated by $m \psi (m)$. Therefore, it follows that the expectation of $D_{n} + D_{n}' - D_{2n}$ can be well approximated by
\begin{equation}\label{sumM}
\sum_{m} m \psi (m) P (M=m).
\end{equation}
By giving a sharp estimate of the distribution of $M$, one can derive precise upper and lower bounds for the sum in \eqref{sumM} and thus obtain \eqref{GOAL1}. For details, we refer to Propositions \ref{upper} and \ref{lower} below.

All error terms arising in this analysis must be controlled so as to be of order at most that of the error term appearing in \eqref{GOAL1}. One difficulty imposed by this requirement is estimating the distribution of $M$ with sufficient precision. The problem of estimating the distribution of $M$ can be reduced to a question concerning a certain type of long-range intersections of the random walk, and is addressed in Proposition \ref{fnk} below.

Finally, we discuss potential future applications of Theorem \ref{main}. As one such application, \cite{AOS} studies the limiting distributions of appropriately rescaled $D_n - \mathbb{E}(D_n)$ and $R_n - \mathbb{E}(R_n)$. In order to carry out such an analysis, precise estimates of $\mathbb{E}(D_{n})$ and $\mathbb{E}(R_{n})$, as provided by Theorem \ref{main}, serve an indispensable role. We also note that such fluctuations of $D_n$ and $R_n$ have been studied in \cite{AO}.
Another direction of application is the study of the simple random walk on the entire trace $\mathcal{G}:= {\cal G}_{0, \infty}$. 
Since $\mathcal{G}$ is locally finite, a simple random walk $X$ can be defined on this random graph. 
A refined estimate of $R_n$ yields volume estimates for resistance balls in 
$\mathcal{G}$ down to small scales, which in turn allows us to derive a sharp heat kernel estimate.
As a further application, our study of collisions of independent simple random walks on $\mathcal{G}$ is in progress and the results will be presented in a separate paper.

The organization of the present paper is as follows. In Section \ref{NOTATION}, we define the notation and concepts that are used throughout the paper. In Section \ref{sec:LRI}, we establish precise probabilistic estimates of long-range intersections of the random walk, which are necessary for evaluating the distribution of $M$. Finally, in Section \ref{pfmain}, we prove Theorem \ref{main}.

\subsection*{Acknowledgments}
The authors thank David A.~Croydon and Umberto De Ambroggio for their helpful discussions. DS was supported by JSPS Grant-in-Aid for Scientific Research (C) 22K03336, JSPS Grant-in-Aid for Scientific Research (B) 22H01128
and 21H00989. SW was supported by JSPS KAKENHI Grant Number 25K23335.

\section{Notation and estimates}

In this section, we introduce notation and provide an improved estimate of the intersection probability of simple random walks on $\mathbb{Z}^4$, which is crucial to prove Theorem \ref{main}. 

\subsection{Notation}\label{NOTATION}

We write $\mathbb{Z}^d$ and $\mathbb{R}^d$  for the $d$-dimensional  integer lattice and the $d$-dimensional Euclidean space, respectively. For $x \in \mathbb{R}^{d}$, $|x|$ denotes the Euclidean distance between $x$ and the origin. We will consider only the $d = 4$ case in this article. 

Let $S$ be a simple random walk in $\mathbb{Z}^{4}$, and write $S(j)$ for the walker at time $j\in \mathbb{Z}_+$, the latter denoting the set of non-negative integers. The probability law and its expectation of $S$ are denoted by $P^{x}$ and $E^{x}$, respectively when we assume $S(0) = x$. Throughout this paper, even when $r \ge 0$ is not necessarily an integer, for notational simplicity (even if it were needed), we do not use $\lfloor r \rfloor$. For instance, when $r$ is not an integer, $S(r)$ is to be understood as $S(\lfloor r \rfloor)$.

Throughout the paper, $c$, $C$, $c_{1}, \cdots$ denote constants whose values may change from line to line. Let $\{ a_{n} \}_{n \ge 1}$ be a sequence with  $a_{n} > 0$ for all $n \ge 1$.
For a sequence $\{ b_{n} \}_{n \ge 1}$, we write $b_{n} = o (a_{n})$ if $\lim_{n \to \infty} \frac{b_{n}}{a_{n}} = 0$.
We also write $b_{n} = O (a_{n})$ if there exists a constant $C > 0$ such that $|b_{n}| \le C a_{n}$ for all $n \ge 1$. When we want to make it explicit that the rate of convergence to $0$ or the constant $C>0$ depends on a certain parameter $\alpha$, we write $o_{\alpha}(a_n)$ or $O_{\alpha}(a_n)$.

We denote by ${\cal G}_{m, n}$ the (random) graph whose vertex set is given by the sites visited by the walker between times $m$ and $n$, and whose edge set corresponds to a pair of consecutive visited sites. Formally, the vertex and edge sets of ${\cal G}_{m, n}$ are
\begin{equation}\label{gmn}
\{ S(j) \mid m \le j \le n \} \text{ and } \{ \{ S(j), S(j+1) \} \mid m \le j \le n-1 \},
\end{equation}
respectively. We also denote by $\mathcal{G}$ the infinite (random) graph whose vertex and edge sets are given by
\[
    \{S(j)\mid j\ge 0\};\quad \{\{S(j),S(j+1)\}\mid j\ge 0\},
\]
respectively. 

For a graph $G = (V, E)$, let $R_{G} (\cdot, \cdot )$ denote the effective resistance  on $G$. That is, $R_{G}(\cdot,\cdot)$ is defined as follows.

\begin{itemize}
\item For $f, g \in \mathbb{R}^{V}$, write
$${\cal E} (f, g) = \frac{1}{2} \sum_{x, y \in V, \, \{ x, y \} \in E} (f(x) - f(y)) (g(x) - g(y)), $$
for a quadratic form. 

\item For  disjoint subsets $A$ and $B$  of $V$, define 
$$ R_{G} (A, B) = \big( \inf \big\{ {\cal E} (f, f) \ : \ f \in \mathbb{R}^{V}, \, {\cal E} (f, f) < \infty, \, f|_{A} = 1, \, f|_{B} = 0 \big\} \big)^{-1}. $$ We write $R_{G} (x, y) = R_{G} (\{ x \}, \{ y \})$ for $x, y \in V$.
\end{itemize}
Also, we let $d_{G} (\cdot, \cdot )$ denote the  graph distance on $G$. 

We say $0 \le k \le n$ is a cut time up to $n$ if $S[0,k] \cap S[k+1, n ] = \emptyset.$ Here, for $0 \le a \le b < \infty$, we denote by $S[a,b]$ the set $\{ S_k : a \le k \le b \}.$
It is shown in \cite{slowly} that there exists $c >0$ such that if we write 
$${\cal C}_{n} = \sharp \big\{ 0 \le k \le n \ : \ k \text{ is a cut time up to } n \big\},$$
(here $\sharp A$ denotes the number of elements in the set $A$), then for $d=4$ 
\begin{equation}\label{LCTN}
E ({\cal C}_{n}) = c n (\log n)^{-\frac{1}{2}} \big( 1 + o (1) \big).
\end{equation}
Note that 
\begin{equation}\label{LCTN1}
{\cal C}_{n} \le R_{{\cal G}_{0,n}} (0, S(n)) \le d_{{\cal G}_{0,n}} (0, S(n)).
\end{equation}
It should be also noted that the quantities $R_{n}$ and $D_{n}$ in Theorem \ref{main} coincide with $R_{{\cal G}_{0,n}} (0, S(n))$ and $d_{{\cal G}_{0,n}} (0, S(n))$, respectively.

\subsection{Long range intersection}\label{sec:LRI}

Fix an integer $k \ge 2$. It is known that as $n \to \infty$,
\begin{equation}\label{LRI}
P \Big( S[0, n] \cap S[(k+1)n, (k+2)n] \neq \emptyset \Big) = \frac{2 \log (k+1) - \log k - \log (k+2)}{2 \log n } \big( 1 + o_{k} (1) \big).
\end{equation}
See the proof of  \cite[Theorem 4.3.6 (i)]{Lawb} for this. 
Unfortunately, in \cite{Lawb}, it does not appear that the order of the error term $o_{k}(1)$ in \eqref{LRI} is specified. %
In this subsection, we will examine the order of $o_{k}(1)$ and its dependence on $k$ (see Proposition \ref{fnk} below). 
As will become clear later, these are necessary to prove Theorem \ref{main}. 

For that purpose, we first evaluate the non-intersection probability of three independent random walks in Corollary \ref{Fn} below. 
Let $S^{1}, S^{2}$ and $ S^{3}$ 
be independent simple random walks in $\mathbb{Z}^{4}$ started at the origin. 

We proceed with this as follows. For $\lambda \in (0,1)$, we prepare independent killing times $T^{1}, T^{2}, T^{3}$ with killing rate $1-\lambda$ (the random times $T^{1}, T^{2}, T^{3}$ are independent of $S^{1}, S^{2}, S^{3}$), that is, for each $j \ge 0$, 
$$ P ( T^{i} = j ) = (1- \lambda) \lambda^{j} \ \ \ \text{for }  \ \  i=1,2,3.$$

Let 
\begin{equation}\label{q}
q= q_{\lambda} = P_{1} \Big( S^{1} (0, T^{1}] \cap \big( S^{2} [0, T^{2}] \cup S[0, T^{3}] \big) = \emptyset \Big)
\end{equation}
be the non-intersection probability given $S^{2}$ and $S^{3}$ so that $q$ is a random variable with respect to $S^{2}, S^{3}$ (and $T^{2}, T^{3}$). That is, the probability $P_{1}$ in \eqref{q} is taken with respect to $S^{1}$ and $T^{1}$. We also write 
\begin{equation*}
J = J_{\lambda} = {\bf 1} \{ 0 \not\in S^{3} (0, T^{3}] \},
\end{equation*}
for the indicator function of the event that $S^{3}$ does not hit the origin up to time $T^{3}$. Finally, we let 
\begin{equation*}
G^{\lambda} = \sum_{j=0}^{T^{2}} G_{\lambda} \big( S^{2} (j) \big) +  \sum_{k=1}^{T^{3}} G_{\lambda} \big( S^{3} (k) \big).
\end{equation*}
Here $G_{\lambda}(x)$ stands for Green's function at $x$ up to the killing time, that is, for  $x \in \mathbb{Z}^{4}$
$$G_{\lambda} (x) = \sum_{j=0}^{\infty} \lambda^{j} P (S_j = x), $$
where $S$ is a simple random walk in $\mathbb{Z}^{4}$ started at the origin.
Namely, $G^{\lambda}$ is a function of $S^{2}[0, T^{2}]$ and $S^{3}[0, T^{3}]$.

Then the last exit decomposition and the translation invariance give that 
\begin{equation}\label{equal1}
E_{2,3} \big( q J G^{\lambda} \big) = 1,
 \end{equation}
see  \cite[Theorem 3.6.1]{Lawb} for this. Here, note that in \eqref{equal1} the expectation $E_{2,3}$ is taken with respect to $S^{2}$ and $S^{3}$ as well as $T^{2}$ and $T^{3}$.

 We aim to pull $G^{\lambda}$ out of the expectation in \eqref{equal1} in the following proposition. For this, \cite[Lemma 4.2.1]{Lawb} shows that as $\lambda \to 1-$, 
\begin{align}
&E_{2,3} (G^{\lambda}) = - \frac{8}{\pi^{2}} \log (1-\lambda) + O (1), \label{ex-var1} \\
&\text{Var} (G^{\lambda}) = O \big( - \log (1- \lambda) \big). \label{ex-var2}
\end{align}
Here, we note that in \eqref{ex-var2} the variance $\text{Var}$ is taken with respect to $S^{2}$ and $S^{3}$ as well as $T^{2}$ and $T^{3}$.
It follows from this and a straightforward application of Chebyshev's inequality that $G^{\lambda}$ is accurately well approximated by its expectation as $\lambda \to 1-$.  Indeed, by combining several more refined estimates, we can establish the following proposition, which ensures that $E_{2,3} (qJ)$ is well approximated by $1/ E_{2,3} (G^{\lambda})$. 

\begin{proposition}\label{qJ}
As $\lambda \to 1-$, we have
\begin{equation*}
E_{2,3} (q J) = \frac{\pi^{2}}{8} \big[ - \log (1-\lambda) \big]^{-1} + O \Big( \big[ - \log (1 -\lambda ) \big]^{-  \frac{28}{25}} \Big). 
\end{equation*}
\end{proposition}

\begin{proof}
We will first consider an upper bound on $E_{2,3} (q J)$.  Write $G^{(2)} =  \sum_{j=0}^{T^{2}} G_{\lambda} \big( S^{2} (j) \big)$ and $G^{(3)} = \sum_{k=1}^{T^{3}} G_{\lambda} \big( S^{3} (k) \big)$ so that $G^{\lambda} = G^{(2)} + G^{(3)}$. Note that $G^{(2)}-1$ and $G^{(3)}$ are independent and identically distributed. Let 
\begin{equation*}
a_{(4)} \coloneqq \frac{8}{\pi^{2}}.
\end{equation*}

Take $\delta \in (0,1)$. For $i=2,3$, we define the event $A_{i}$ by 
$$ A_{i} = \Big\{ |G^{(i)} + \frac{a_{(4)}}{2} \log (1- \lambda) | \ge \big[ -   \log (1 -\lambda ) \big]^{1- \delta} \Big\}.$$
Combining \eqref{ex-var1} and \eqref{ex-var2}, Markov's inequality shows that 
\begin{equation}\label{ai}
P (A_{i} ) \le \frac{C}{ \big[ -   \log (1 -\lambda ) \big]^{1- 2 \delta}},
\end{equation}
for $i=2, 3$. For a random variable $X$ and an event $A$, if we write $E(X {\bf 1}_{A})$ as $E(X; A)$, \eqref{ai} yields that 
\begin{align*}
E_{2,3} (q J) &= E_{2,3} (q J ; A_{2} \cup A_{3} ) + E_{2,3} (q J ; A_{2}^{c} \cap A_{3}^{c} ) \\
&\le E_{2,3} (q J ; A_{2} \cup A_{3} )  + \Big\{ - a_{(4)} \log (1- \lambda) - 2 \big[ -   \log (1 -\lambda ) \big]^{1- \delta} \Big\}^{-1} E_{2,3} (qJ G^{\lambda}) \\
&= E_{2,3} (q J ; A_{2} \cup A_{3} )  + \Big\{ - a_{(4)} \log (1- \lambda) - 2 \big[ -   \log (1 -\lambda ) \big]^{1- \delta} \Big\}^{-1} \\
&\le E_{2,3} (q J ; A_{2}  ) + E_{2,3} (q J ; A_{3} ) + \frac{\pi^{2}}{8} \big[ - \log (1-\lambda) \big]^{-1} + O \Big( \big[ - \log (1 -\lambda ) \big]^{-1-\delta} \Big).
\end{align*}
However, since the event $A_{2}$ is independent of $S^{1}$ and $S^{3}$, we see that 
\begin{align*}
E_{2,3} (qJ ; A_{2} ) &\le P \big( S^{1} (0, T^{1} ] \cap S^{3} [0, T^{3} ] = \emptyset \big) P (A_{2} ) \\
&\le C \big[ - \log (1- \lambda) \big]^{-\frac{1}{2}}  \big[ -   \log (1 -\lambda ) \big]^{-1 +  2 \delta} \\
&= C \big[ - \log (1- \lambda) \big]^{-\frac{3}{2} + 2 \delta}, 
\end{align*}
where we used \cite[Corollary 4.2.3]{Lawb} and the Cauchy-Schwarz inequality to show that 
$$ P \big( S^{1} (0, T^{1} ] \cap S^{3} [0, T^{3} ] = \emptyset \big) \le C \big[ - \log (1- \lambda) \big]^{-\frac{1}{2}}.$$
Since a similar estimate holds for $E_{2,3} (q J ; A_{3} )$,  we can conclude that 
\begin{equation}\label{Upper}
E_{2,3} (q J) \le \frac{\pi^{2}}{8} \big[ - \log (1-\lambda) \big]^{-1} + O \Big( \big[ - \log (1 -\lambda ) \big]^{-1-\delta} \Big) + O \Big( \big[ - \log (1 -\lambda ) \big]^{- \frac{3}{2} + 2 \delta} \Big).
\end{equation}

As for the lower bound on $E_{2,3} (q J)$, since $0 \le q J {\bf 1}_{A_{2} \cup A_{3}} \le 1$, H\"{o}lder's inequality gives that 
\begin{align*}
E_{2,3} (q J G^{\lambda} {\bf 1}_{A_{2} \cup A_{3}} ) &\le \big\{ E_{2,3} ((G^{\lambda})^{9}) \big\}^{\frac{1}{9}} \big\{ E_{2,3} (q J {\bf 1}_{A_{2} \cup A_{3}} ) \big\}^{\frac{8}{9}} \\
& \le C \big[ - \log (1- \lambda) \big]^{-\frac{4}{3} +  \frac{16 \delta}{9}} \big\{ E_{2,3} ((G^{\lambda})^{9}) \big\}^{\frac{1}{9}}.
\end{align*}
On the other hand, as claimed in \cite[line -10 in page 119]{Lawb}, a routine calculation shows that 
$$ \big\{ E_{2,3} ((G^{\lambda})^{9}) \big\}^{\frac{1}{9}} = O \big( - \log (1- \lambda) \big). $$
(The details are left to the reader.) This gives 
$$ E_{2,3} (q J G^{\lambda} {\bf 1}_{A_{2} \cup A_{3}} ) \le C \big[ - \log (1- \lambda) \big]^{-\frac{1}{3} +  \frac{16 \delta}{9}}.$$
Therefore, since $E_{2,3} (qJG^{\lambda}) = 1$, it follows that 
\begin{align*}
1 - O \Big( \big[ - \log (1- \lambda) \big]^{-\frac{1}{3} +  \frac{16 \delta}{9}} \Big) &= E_{2,3} (q J G^{\lambda} ; A_{2}^{c} \cap A_{3}^{c} ) \\
&\le \Big\{ - a_{(4)} \log (1- \lambda) + 2 \big[ -   \log (1 -\lambda ) \big]^{1- \delta} \Big\} E_{2,3} (q J).
\end{align*}
So we have
\begin{equation}\label{Lower}
E_{2,3} (qJ) \ge \frac{\pi^{2}}{8} \big[ - \log (1-\lambda) \big]^{-1} - O \Big(  \big[ - \log (1 -\lambda ) \big]^{-1-\delta} \Big) - O \Big( \big[ - \log (1 -\lambda ) \big]^{- \frac{4}{3} + \frac{16 \delta}{9}} \Big).
\end{equation}
To optimize these estimates, by choosing $\delta = \frac{3}{25}$, all error terms $O$ in \eqref{Upper} and \eqref{Lower} can be bounded from above by 
$$O \Big(  \big[ - \log (1 -\lambda ) \big]^{- \frac{28}{25} } \Big).$$
This finishes the proof.
\end{proof}

Let 
\begin{equation*}
F_{n} \coloneqq \big\{ S^{1} (0, n] \cap \big( S^{2}[0, n] \cup S^{3} [0, n] \big) = \emptyset, \ \ 0 \not\in S^{3} (0, n] \big\},
\end{equation*}
be the non-intersection event for three independent random walks up to $n$ steps.

\begin{corollary}\label{Fn}
As $n \to \infty$, we have 
\begin{equation*}
P(F_{n} ) = \frac{\pi^{2}}{8} (\log n)^{-1} + O \Big( (\log n)^{-\frac{28}{25}} \Big).
\end{equation*}
\end{corollary}

\begin{proof}
Write 
\begin{equation*}
F_{\lambda} \coloneqq \big\{ S^{1} (0, T^{1}] \cap \big( S^{2}[0, T^{2}] \cup S^{3} [0, T^{3}] \big) = \emptyset, \ \ 0 \not\in S^{3} (0, T^{3}] \big\},
\end{equation*}
and set $\lambda_{n} = 1 - n^{-1}$.  Note that if $T$ is the killing time with killing rate $\lambda_{n}$, by a simple estimate for the geometric distribution
\begin{equation*}
P \big( T \ge n (\log n)^{2} \big) \le C e^{-c (\log n)^{2}} \ \ \text{ and }  \ \   P \big( T \le n (\log n)^{-2} \big) \le (\log n)^{-2}, 
\end{equation*}
for some universal constants $c, C \in (0, \infty)$.
Thus, using Proposition \ref{qJ} and the monotonicity of $F_n$ with respect to $n$, we have 
\begin{align*}
\frac{\pi^{2}}{8} ( \log n)^{-1} + O \Big(   (\log n )^{- \frac{28}{25} } \Big) &= P ( F_{\lambda_{n}} )  \\
&\ge P \big( T^{1} \vee T^{2} \vee T^{3} \le n (\log n)^{2}, \  F_{n (\log n)^{2}} \big) \\
&\ge P \big( F_{n (\log n)^{2}} \big) - C e^{-c (\log n)^{2}}.
\end{align*}
Reparametrizing this, we have 
$$ P ( F_{n} ) \le \frac{\pi^{2}}{8} ( \log n)^{-1} + O \Big(   (\log n )^{- \frac{28}{25} } \Big).$$
On the other hand, by Proposition \ref{qJ} and the monotonicity of $F_n$ with respect to $n$ again, it follows that 
\begin{align*}
\frac{\pi^{2}}{8} ( \log n)^{-1} - O \Big(   (\log n )^{- \frac{28}{25} } \Big) &= P ( F_{\lambda_{n}} )  \\
&\le P \big( T^{1} \wedge T^{2} \wedge T^{3} \ge n (\log n)^{-2}, \  F_{n (\log n)^{-2}} \big) + 3 (\log n)^{-2} \\
&\le P \big( F_{n (\log n)^{-2}} \big) + 3 (\log n)^{-2}.
\end{align*}
This gives
$$P ( F_{n} ) \ge \frac{\pi^{2}}{8} ( \log n)^{-1} - O \Big(   (\log n )^{- \frac{28}{25} } \Big),$$
which finishes the proof.
\end{proof}

Take $k \in \mathbb{N}$. In the next proposition, we will give a sharp asymptotic expression for 
\begin{equation*}
f(n;k) \coloneqq P \big( S^{1} [0, n] \cap S^{2} [kn, (k+1)n ] \neq \emptyset \big).
\end{equation*}

\begin{proposition}\label{fnk}
The following equation holds for all $n, k \ge 2$.
\begin{align*}
f (n;k) = \frac{1}{2 \log n} \, \log \Big[ 1 + \frac{1}{k^{2} +2k } \Big] \,  \bigg\{ 1 + O \bigg( \frac{1}{k(\log n)^{\frac{3}{25}}} \bigg) \bigg\},
\end{align*}
where the constants used in $O$ are universal and do not depend on $n$ and $k$. 
\end{proposition}

\begin{proof}
In what follows, we modify the method developed in \cite[Section 4.3]{Lawb} to suit our purposes. To this end, we begin by looking at the ``first intersection" of $S^{1}$ and $S^{2}$ as follows. Let 
\begin{align*}
&\tau \coloneqq \inf \{ j \ge 0 \ | \ S^{1} (j) \in S^{2} [kn, (k+1)n ] \}, 
\\
&\sigma \coloneqq \inf \{ l \ge kn \ | \ S^{2} (l) = S^{1} (\tau) \}. 
\end{align*}
Writing $\Gamma (j, l) \coloneqq P (\tau = j, \sigma = l)$, we have 
\begin{equation}\label{eq:sumgamma}
    f(n;k) = P (\tau \le n ) = \sum_{j=0}^{n} \sum_{l=kn}^{(k+1)n} \Gamma (j,l).
\end{equation}

Translating $S^{1} (\tau) = S^{2} (\sigma)$ to the origin and reversing both $S^{1} [0, \tau]$ and $ S^{2}[0, \sigma]$, we see that 
\begin{align*}
&\Gamma (j,l) = P \Big( S^{1} (0, j] \cap \big( S^{2} [0, (k+1)n -l ] \cup S^{3} [0,  l-kn]  \big) = \emptyset, \\
& \ \ \ \ \ \ \ \ \ \ \ \ \ \ \ \ \ \ \ \ \ \ 0 \not\in S^{3} (0,  l-kn], \ S^{1} (j) = S^{3} (l) \Big).
\end{align*}
We will show that by choosing $a > 0$ appropriately, the case that $j \wedge ( l-kn ) \wedge \{ (k+1)n -l \}  \le n (\log n)^{-a}$ is negligible. If all of these three times are larger than $n (\log n)^{-a}$, we can apply Corollary \ref{Fn} to show that the probability of the event 
$$A^1:=\Big\{  S^{1} (0, j] \cap \big( S^{2} [0, (k+1)n -l ] \cup S^{3} [0,  l-kn]  \big) = \emptyset, \ 0 \not\in S^{3} (0,  l-kn] \Big\},$$
is bounded above by
$$\frac{\pi^{2}}{8} (\log n)^{-1} + O \Big( (\log n)^{-\frac{28}{25}} \Big).$$

Suppose that $j \wedge ( l-kn ) \wedge \{ (k+1)n -l \}  \ge n (\log n)^{-a}$ and $j + l$ is even. For $x = (x_{1}, \cdots, x_{4}) \in \mathbb{Z}^{4}$, we write $0 \leftrightarrow x$ if $x_{1} + \cdots x_{4}$ is even. For $j, l \in \mathbb{Z}$, we also write $j \leftrightarrow l$ if $j+l$ is even. Then the Markov property at time $n (\log n)^{-2a}$ (assuming that $n (\log n)^{-2a} \in \mathbb{Z}$) yields that 
\begin{align*}
&\Gamma (j,l) \\
&\le P \Big( S^{1} (0, n (\log n)^{-2a}] \cap \big( S^{2} [0, n (\log n)^{-2a} ] \cup S^{3} [0,  n (\log n)^{-2a}]  \big) = \emptyset, \\
& \ \ \ \ \ \ \  \ \ \ \ 0 \not\in S^{3} (0,  n (\log n)^{-2a}] \Big)   \times \max_{x \in \mathbb{Z}^{4}, \, 0 \leftrightarrow x} P \Big( S \big( j + l - 2n (\log n)^{-2a} \big) = x \Big) \\
&\le P \Big( S^{1} (0, n (\log n)^{-2a}] \cap \big( S^{2} [0, n (\log n)^{-2a} ] \cup S^{3} [0,  n (\log n)^{-2a}]  \big) = \emptyset, \\
& \ \ \ \ \ \ \ \ \  \ \ 0 \not\in S^{3} (0,  n (\log n)^{-2a}] \Big)  \times P \big( S(j+l) =0 \big) \, \Big\{ 1 + O \Big[ \frac{(\log n)^{-2a}}{k} \Big] \Big\} \\
&= \Big\{ \frac{\pi^{2}}{8} (\log n)^{-1} + O \Big( (\log n)^{-\frac{28}{25}} \Big) \Big\} \, \frac{8}{\pi^{2}} (j+l)^{-2} \, \Big\{ 1 + O \Big[ \frac{(\log n)^{-2a}}{k} \Big] \Big\} \\
&= (\log n)^{-1} (j+l)^{-2}  \, \Big\{ 1 + O \Big( (\log n)^{-\frac{3}{25}} \Big) \Big\} \, \Big\{ 1 + O \Big[ \frac{(\log n)^{-2a}}{k} \Big] \Big\},
\end{align*} 
where we used  \cite[Proposition 1.2.5]{Lawb} to see that  
\begin{align*}
\max_{x \in \mathbb{Z}^{4}, \, 0 \leftrightarrow x} P \Big( S \big( j + l - 2n (\log n)^{-2a} \big) = x \Big) &\le  P \big( S(j+l) = 0 \big) \, \Big\{ 1 + O \Big[ \frac{(\log n)^{-2a}}{k} \Big] \Big\} \\
&= \frac{8}{\pi^{2}} (j+l)^{-2} \, \Big\{ 1 + O \Big[ \frac{(\log n)^{-2a}}{k} \Big] \Big\}.
\end{align*}
For the sum in \eqref{eq:sumgamma}, it suffices to consider only the pairs $(j, l)$ for which $j \leftrightarrow l$. Thus, by choosing $a=10$ and taking the sum,  we have 

\begin{align*}
f (n;k) &\le \frac{1}{2 \log n} \, \log \Big[ 1 + \frac{1}{k^{2} +2k } \Big] \, \Big\{ 1 + O \Big( (\log n)^{-\frac{3}{25}} \Big) \Big\} \, \Big\{ 1 + O \Big[ \frac{(\log n)^{-2a}}{k} \Big] \Big\} \\
& \ \ \ \ + \sum \, \Gamma(j,l) \\
&\le  \frac{1}{2 \log n} \, \log \Big[ 1 + \frac{1}{k^{2} +2k } \Big] \, \Big\{ 1 + O \Big( (\log n)^{-\frac{3}{25}} \Big) \Big\} \, \Big\{ 1 + O \Big[ \frac{(\log n)^{-2a}}{k} \Big] \Big\} \\
&\ \ \ \  \ + O \Big( \frac{ (\log n)^{-a} }{k^{2}} \Big) \\
&= \frac{1}{2 \log n} \, \log \Big[ 1 + \frac{1}{k^{2} +2k } \Big] \,  \Big\{ 1 + O \Big( (\log n)^{-\frac{3}{25}} \Big) \Big\} \, \Big\{ 1 + O \Big[ \frac{(\log n)^{-2a}}{k} \Big] \Big\},
\end{align*}
where the sum $\sum$ appearing on the right-hand side of the first inequality above is taken over all pairs $(j,k)$ satisfying $j \wedge ( l-kn ) \wedge \{ (k+1)n -l \}  \le n (\log n)^{-a}$. In the second inequality, we used both the fact that the number of such pairs is at most $C n^{2} (\log n)^{-a}$ and that $\Gamma(j,l) \le C (kn)^{-2}$ holds for these pairs $(j,l)$.
This gives an improved upper bound of \eqref{LRI}. 

We will next deal with a lower bound. For that purpose, we define 
\begin{align*}
A^1_{n}:&=\Big\{  S^{1} (0, n (\log n)^{-2a}] \cap \big( S^{2} [0,  n (\log n)^{-2a}] \cup S^{3} [0,  n (\log n)^{-2a}]  \big) = \emptyset, \\
& \ \ \ \ \ \ \  \ \ 0 \not\in S^{3} (0,  n (\log n)^{-2a}] \Big\},
\end{align*}
and $A^{2}_{j,l} = \{ S^{1} (j) = S^{3} (l) \}$. Suppose that $j \wedge ( l-kn ) \wedge \{ (k+1)n -l \}  \ge n (\log n)^{-a}$ and $j \leftrightarrow l$. We need to estimate 
$$P \big( A^1_{n} \cap (A^{1})^{c} \cap A^{2}_{j,l} \big).$$

To do it, setting $\alpha_{n} = n (\log n)^{-2a}$, we note that the following inequalities hold.
\begin{align*}
 &\text{(i)} \ P \Big( 0 \in S^{3} [\alpha_{n}, l- kn], \, A^{2}_{j,l} \Big) \le C \frac{ (\log n)^{4a}}{k^{2} n^{3}}, \\
&\text{(ii)} \ P \Big( A^{1}_{n}, \, S^{1} (0, \alpha_{n}] \cap S^{2} [\alpha_{n}, (k+1)n-l] \neq \emptyset, \, A^{2}_{j,l} \Big) \\
& \  \ \ \ \ \ \le P \Big( A^{1}_{n}, \, S^{1} (0, \alpha_{n}] \cap S^{2} [\alpha_{n}, (k+1)n-l] \neq \emptyset \Big) \max_{x \in \mathbb{Z}^{4}} P \big( S^{3} (l- \alpha_{n} ) = x \big) \\
&\ \ \ \ \  \ \le \frac{C}{k^{2} n^{2}} \frac{(\log \log n)^{4}}{(\log n)^{2}} \ \ \  (\text{this inequality follows from   \cite[Lemma 5.4]{slowly}}), \\
&\text{(iii)} \  P \Big( A^{1}_{n}, \, S^{2} [0, \alpha_{n}] \cap S^{1} [\alpha_{n}, j] \neq \emptyset, \, A^{2}_{j,l} \Big) \le \frac{C}{k^{2} n^{2}} \frac{(\log \log n)^{4}}{(\log n)^{2}} \ \ \ \  \\ 
& \ \ \ \ \ \ \ \ \ \ \ \ \ (\text{for the same reason as in (ii)}), \\
&\text{(iv)} \ P \Big( A^{1}_{n}, \, S^{1} [\alpha_{n}, j] \cap S^{2} [\alpha_{n}, (k+1)n -l] \neq \emptyset, \, A^{2}_{j,l} \Big) \le \frac{C}{k^{2} n^{2}} \frac{(\log \log n)^{4}}{(\log n)^{2}} \\
& \ \ \ \ \ \ \ \ \ \ \ \ \ (\text{for the same reason as in (ii)}).
\end{align*}
By a similar argument, we obtain the following inequality.
$$P \Big( A^{1}_{n} \cap \big\{ S^{1} (0, j] \cap S^{3} [0, l-kn] \neq \emptyset \big\} \cap A^{2}_{j,l} \Big) \le \frac{C}{k^{2} n^{2}} \frac{(\log \log n)^{4}}{(\log n)^{2}}.$$
Combining these estimates, we have 
$$P \big( A^1_{n} \cap (A^{1})^{c} \cap A^{2}_{j,l} \big) \le \frac{C}{k^{2} n^{2}} \frac{(\log \log n)^{4}}{(\log n)^{2}}.$$

Therefore, for $(j,l)$ satisfying that $j \wedge ( l-kn ) \wedge \{ (k+1)n -l \}  \ge n (\log n)^{-a}$ and $j \leftrightarrow l$, we have
\begin{align*}
\Gamma(j,l) = P (A^{1} \cap A^{2}_{j,l}) &= P (A^{1}_{n} \cap A^{2}_{j,l} ) - P \big( A^1_{n} \cap (A^{1})^{c} \cap A^{2}_{j,l} \big) \\
&\ge  P (A^{1}_{n} \cap A^{2}_{j,l} ) - \frac{C}{k^{2} n^{2}} \frac{(\log \log n)^{4}}{(\log n)^{2}}.
\end{align*}
Recall that we take $a =10$. Since the probability of the event that 
$$|S^{1} (\alpha_{n})| \wedge |S^{3} (\alpha_{n})| \ge \sqrt{n} (\log n)^{-\frac{a}{2}},$$
 is smaller than $ C n^{-4}$, using the Markov property at time $\alpha_{n}$ and \cite[Proposition 1.2.5]{Lawb} again, we see that for $(j,l)$ satisfying that $j \wedge ( l-kn ) \wedge \{ (k+1)n -l \}  \ge n (\log n)^{-a}$ and $j \leftrightarrow l$,
 \begin{align*}
  P (A^{1}_{n} \cap A^{2}_{j,l} ) &\ge P (A^{1}_{n} ) \times \min  P \Big( S \big( j + l - 2n (\log n)^{-2a} \big) = x \Big)  - C n^{-4} \\
  &\ge P (A^{1}_{n} ) P \big( S(j+l) = 0 \big) \Big\{ 1 - O \Big( (\log n)^{-a} \Big) - O \Big[ \frac{(\log n)^{-2a}}{k} \Big] \Big\} - C n^{-4} \\
  &=  (\log n)^{-1} (j+l)^{-2}  \, \Big\{ 1 - O \Big( (\log n)^{-\frac{3}{25}} \Big) \Big\} \, \Big\{ 1 - O \Big[ \frac{(\log n)^{-2a}}{k} \Big] \Big\},
  \end{align*}
where the minimum in the first inequality is taken over all $x \in \mathbb{Z}^{4}$ satisfying $0 \leftrightarrow x$ and $|x| \le \sqrt{n} (\log n)^{-\frac{a}{2}}$. Since the last quantity above is comparable to $(\log n)^{-1} (kn)^{-2}$,  this yields 
$$\Gamma(j,l) \ge (\log n)^{-1} (j+l)^{-2}  \, \Big\{ 1 - O \Big( (\log n)^{-\frac{3}{25}} \Big) \Big\} \, \Big\{ 1 - O \Big[ \frac{(\log n)^{-2a}}{k} \Big] \Big\}.$$

Thus, by taking the sum for all pairs $(j,l)$ with $j \wedge ( l-kn ) \wedge \{ (k+1)n -l \}  \ge n (\log n)^{-a}$ and $j \leftrightarrow l$, it follows that 
\begin{align*}
f (n;k) \ge  \frac{1}{2 \log n} \, \log \Big[ 1 + \frac{1}{k^{2} +2k } \Big] \,  \Big\{ 1 - O \Big( (\log n)^{-\frac{3}{25}} \Big) \Big\} \, \Big\{ 1 - O \Big[ \frac{(\log n)^{-2a}}{k} \Big] \Big\},
\end{align*}
as desired.
\end{proof}

\section{Proof of Theorem \ref{main}}\label{pfmain}

In this section, we prove Theorem \ref{main}. 
We give a preliminary lemma on the asymptotic behavior of a positive function in Section \ref{section:lemma}. 
We then observe that the functions that characterize $E[D_n]$ and $E[R_n]$ both exhibit such behavior in Section \ref{section:check}. 

\subsection{Preliminary lemma}\label{section:lemma}

\begin{lemma}\label{deriv}
Suppose that a function $\psi : \mathbb{N} \to (0, \infty)$ satisfies the following two conditions: 

\begin{itemize}
\item[(I)] there exists $c >0$ such that for all $n \ge 1$ and $k \in [n, 2n]$ 
$$\psi (k) = \psi (n) \big\{ 1 + O \big( (\log n)^{-c} \big) \big\}.$$

\item[(II)] there exists $\epsilon > 0$ such that for all $n \ge 1$ 
$$ \psi (2n) = \psi (n) \Big\{ 1 - \frac{\log 2 }{2 \log n} + O \big( (\log n)^{-1 - \epsilon} \big) \Big\}.$$ 

\end{itemize}
Set $\delta = \min \{ c, \epsilon \}$. Then there exists $a > 0$ such that 
$$\psi (n) = a  (\log n)^{-\frac{1}{2}} \Big\{ 1 + O \Big( \big( \log n \big)^{-\delta} \Big) \Big\},$$

\end{lemma}

\begin{proof}
Let $g(n) = \log \psi (2^{n})$. By the second condition (II),
$$g(n) = g(n-1) - \frac{1}{2n} + O \big( n^{-1-\epsilon} \big).$$
This gives  that for $m < n$
$$g(n) -g(m) = -\frac{1}{2} \log n + \frac{1}{2} \log m + O (m^{-\epsilon} ).$$
Thus, there exists $b \in \mathbb{R}$ such that 
$$\lim_{n \to \infty} \Big[ g(n) + \frac{1}{2} \log n \Big] = b \ \ \ \text{and} \ \ \ g(n) = -\frac{1}{2} \log n + b + O (n^{-\epsilon}).$$
This yields 
$$\psi (2^{n}) = e^{b} n^{-\frac{1}{2}} \big\{ 1 + O (n^{-\epsilon}) \big\} = \sqrt{\log 2} \, e^{b} (\log 2^{n})^{-\frac{1}{2}} \Big\{ 1 + O \Big( \big( \log 2^{n} \big)^{-\epsilon} \Big) \Big\}.$$
Let $a = \sqrt{\log 2} \, e^{b}$ and $\delta = \min \{ c, \epsilon \}$. For $k > 1$, taking $n$ satisfying $k \in [2^{n}, 2^{n+1}]$ and using first condition (I), we have 
$$\psi (k) = a  (\log k)^{-\frac{1}{2}} \Big\{ 1 + O \Big( \big( \log k \big)^{-\delta}\Big) \Big\},$$
as desired.
\end{proof}

\subsection{Checking the conditions (I) and (II)}\label{section:check}
We write $R_{n} =R_{{\cal G}_{0,n}} (0, S(n))$ and $D_{n} = d_{{\cal G}_{0,n}} (0, S(n))$. 
In this section, we will show that $\psi$ satisfies the conditions (I) and (II) appearing in Lemma \ref{deriv}. All the results obtained for $D_n$ in this section also hold for $R_n$. Therefore, in this section, we will consider only $D_n$.
Let 
\begin{equation}\label{psi}
\psi (n) = \frac{E ( D_{n} )}{n}.
\end{equation}

\begin{remark}\label{prev}
 In \cite{exact}, it was shown that for any $\epsilon > 0$ there exists $N_{\epsilon}$ such that for all $n \ge N_{\epsilon}$
\begin{equation}\label{prev1}
\frac{E ( R_{2n} )}{2n} \le \frac{E ( R_{n} )}{n} \Big\{ 1 - \frac{\log 2 }{2 \log n} ( 1 - \epsilon)  \Big\}.
\end{equation}
By using \eqref{prev1} and using the number of cut times as a lower bound of $E ( R_{n} )$, it was also shown in \cite{exact} that $E ( R_{n} ) = n (\log n)^{-1/2 + o(1)}$. Since the arguments in \cite{exact} can be safely applied to $D_n$ as well, it follows that for the function $\psi$ defined in \eqref{psi},
\begin{equation*}
 \psi (n) = (\log n)^{-1/2 + o(1)}.
\end{equation*}

To check the condition (II), all error terms must be smaller than $O \Big( \psi (n) (\log n)^{-1-\epsilon} \Big)$ for some $\epsilon>0$.  As we will see below, this constraint is quite sensitive, and this is the reason why we need Proposition \ref{fnk} derived in the previous section. 
\end{remark}

Let 
$$D^{1}_{n} = d_{{\cal G}_{0,\frac{n}{2}}} \Big( 0, S \Big( \frac{n}{2} \Big) \Big) \ \ \text{ and } \ \ D^{2}_{n} = d_{{\cal G}_{\frac{n}{2}}, n} \Big( S \Big( \frac{n}{2} \Big), S(n) \Big),$$
so that $D_{n}^{1}$ and $D_{n}^{2}$ are independent and identically distributed. 
Note that $D^{1}_{n} + D^{2}_{n} - D_{n} $ is always nonnegative by the triangle inequality and 
\begin{equation}\label{eq:Dn-s}
    E (D^{1}_{n} + D^{2}_{n} - D_{n} ) = n \Big\{  \psi \Big( \frac{n}{2} \Big) - \psi (n) \Big\}.
\end{equation}    
Thus, to achieve our goal of checking the assumptions (I) and (II), we need to estimate the expectation on the left-hand side of \eqref{eq:Dn-s}. 

To do it, we first introduce some notation. For $r \in \mathbb{R}$, we write 
$$a_{n,r} = n (\log n)^{r}.$$
Let  $b = \frac{13}{12}$ and $N= (\log n)^{b}$. Divide $[0,n]$ into $N$ intervals  as follows: set
\begin{itemize}
\item $t_{j} = \frac{n}{2}- ja_{n, -b}$ and $I_{j} = \big[ t_{j}, t_{j-1} \big]$ for $j = 1,2, \cdots, \frac{N}{2}$, (notice that $t_{j} < t_{j-1}$) 

\item $t_{k}' = \frac{n}{2} +  k a_{n,-b}$ and  $I_{k}' = \big[ t_{k-1}', t_{k}' \big]$ for $k = 1,2, \cdots, \frac{N}{2}$, (notice that $t_{k-1}' < t_{k}'$)

\item ${\cal G}_{j} = {\cal G}_{t_{j}, t_{j-1}}$  and  ${\cal G}_{k}' = {\cal G}_{t_{k-1}', t_{k}'}$ for $j, k = 1,2, \cdots, \frac{N}{2}$,

\item $D(j) = d_{{\cal G}_{j}} \big( S(t_{j}), S(t_{j-1}) \big)$ and $D(k)'= d_{{\cal G}'_{k}} \big( S(t'_{k-1}), S(t'_{k}) \big)$ for $j, k = 1,2, \cdots, \frac{N}{2}.$
\end{itemize}

For $j, k = 1,2, \cdots, \frac{N}{2}$, we write
$$I_{j} \leftrightarrow I_{k}' := \big\{ S [t_{j}, t_{j-1} ] \cap S [t_{k-1}', t_{k}' ] \neq \emptyset \big\}$$
for the event that ${\cal G}_{j}$ intersects ${\cal G}_{k}'$. Note that the event $I_{1} \leftrightarrow I_{1}'$ always occurs since $S (n/2)$ is a common vertex. Intuitively, if the event $I_{j} \leftrightarrow I_{k}' $ occurs for large $j,k$, then the difference $D^{1}_{n} + D^{2}_{n} - D_{n} $ is big. This is because, due to the intersection $I_{j} \leftrightarrow I_{k}' $, $S[0,n]$ has a shortcut, which makes $D_n$ smaller than $D_{n}^{1} + D_{n}^{2}$.  With this in mind, let
\begin{equation}\label{L}
L = \max \big\{ j+k \ | \ I_{j} \leftrightarrow I_{k}'  \big\}
\end{equation}
denote the size of the longest intersection. Clearly, $2 \le L \le N$. 

We first consider an upper bound of  $E (D^{1}_{n} + D^{2}_{n} - D_{n} )$. The following lemma concerning the variance of $D_n$ is needed.

\begin{lemma}\label{var}
We have 
\begin{equation}\label{lem-var}
    E(D_{n}^{2}) - \big\{ E (D_{n} )\big\}^{2} \le C \big\{ E (D_{n} )\big\}^{2} (\log n)^{-\frac{5}{12}}.
\end{equation}
\end{lemma}

\begin{proof}
Let 
\begin{equation}\label{lem-def-sumstar}
    \sum_\star \coloneqq \sum_{j=1}^{N/2} D(j)+\sum_{k=1}^{N/2}D'(k).
\end{equation}
Then the left-hand side of \eqref{lem-var} is bounded above by 
\begin{align}
	E(\{D_n-E(D_n)\}^2)=E&\left(\left(D_n-\sum_\star+\sum_\star-E\Big(\sum_\star\Big)+E\Big(\sum_\star\Big)-E[D_n]\right)^2\right)	\notag	\\
		\le 3&\Biggl(\left(E\Big(\sum_\star\Big)-E(D_n)\right)^2+\mathrm{Var}\Big(\sum_\star\Big)+E\Big(\Big(D_n-\sum_\star\Big)^2\Big)\Biggr).	\label{lem-var-sum}
\end{align}
We bound from above the right-hand side step by step. We begin with the first term. 

We say that a time $0 \le t \le n$ is a cut time up to time $n$ if $S[0, t] \cap S[t+1, n] = \emptyset$. Furthermore, if $t$ is a cut time up to time $n$, we call $S(t)$ a cut point up to time $n$.
Let $Z_{j}$ be the indicator function of the event that  
\begin{align*}
&\Big\{ \nexists t \in \big[t_{j-1},t_{j-1}+a_{n, -6} \big] \text{ such that $t$ is a cut time up to $n$}  \Big\} \\
& \ \ \ \ \cup \ \Big\{ \nexists t \in \big[t_j-a_{n, -6},t_j\big] \text{ such that $t$ is a cut time up to $n$}  \Big\},
\end{align*}
and define $Z'_k$ by replacing $t_j$ with $t'_k$. 
By  \cite[Lemma 7.7.4]{Lawb}, we have 
\[P(Z_{j} =1) \le C (\log n)^{-1} \log \log n;\quad P(Z'_{k} =1) \le C (\log n)^{-1} \log \log n,\]
for each $j$ and $k$. By decomposing $S[0,n]$ at the cut points into non-intersecting subpaths, we see that 
\begin{align}
    0 \le \sum_\star - D_{n}&\le \sum_{i=1}^{N/2} D(j)Z_j+\sum_{k=1}^{N/2}D'(k)Z'_k+4a_{n,-b}   \label{sum-D}
        \\ &\le a_{n,-b}\left(\sum_{j=1}^{N/2}Z_{j}+\sum_{k=1}^{N/2}Z'_{k}\right) + 4 a_{n,-b}. \notag
\end{align}
In particular, since $b = 13/12$, this ensures that  
\begin{equation}\label{dif}
E \Big( \sum_\star \Big) = E (D_{n}) + O \big( n (\log n)^{-1} \log \log n \big).
\end{equation}
Moreover, from \eqref{dif} it follows that 
\begin{equation}\label{lem-var-1st}
	\left(E\Big(\sum_\star\Big)-E(D_n)\right)^2=O\left(\frac{n^2(\log\log n)^2}{(\log n)^2}\right)\le E(D_n)^2\cdot O\left(\frac{(\log \log n)^2}{\log n}\right),
\end{equation}
since $E(D_{n}) \ge c n (\log n)^{-\frac{1}{2}}$ (see \eqref{LCTN} and \eqref{LCTN1} for this). 

Next, we deal with the second term on the right-hand side of \eqref{lem-var-sum}. 
Let $r \in (0,4)$. As stated in  \cite[Remark 2.1.3]{exact}, if we modify the argument in \eqref{dif} above, we have
\begin{equation}\label{eqiv1}
\psi (a_{n, -r}) = \psi (n) \Big\{ 1 + O \Big( \frac{\log \log n}{\sqrt{\log n}} \Big) \Big\}
\end{equation}
where we used $\psi (n) \ge c (\log n )^{-1/2}$ again. (In \cite{exact}, \eqref{eqiv1} is proved for $E(R_n)/n$, but the same method can be applied to $\psi(n)$ as well.) 
By independence and translation invariance, we have 
\begin{align}
    \mathrm{Var}\Big(\sum_\star\Big)&=\sum_{j=1}^{N/2}\mathrm{Var}(D(j))+\sum_{k=1}^{N/2}\mathrm{Var}(D'(k))    \notag\\
    &\le (\log n)^b E(D(1)^2)    \notag\\
    &\le (\log n)^b a_{n,-b}E\left(D_{a_{n,-b}}\right) \notag   \\
    &\le cn^2(\log n)^{-b}\psi(n)\le cE(D_n)^2(\log n)^{-7/12},\label{var-sumstar}
\end{align}
where we used $D(1)\le a_{n,-b}$. We also used \eqref{eqiv1} in the second-to-last inequality. 

Finally, we deal with the third term. It follows from \eqref{sum-D} that 
\begin{equation}\label{lemDn-third-term}
	E\left(\Big(D_n-\sum_\star\Big)^2\right)\le 2E\Big(\sum_\star\!^{2}\Big)+32E(D_n)^2(\log n)^{-5}.
\end{equation}
Let $\varepsilon\in(0,1)$. We define the event $F$ by 
\[
    F=\bigcap_{j=1}^{N/2}\{D(j)\le (1+\varepsilon)a_{n,-b}\psi(n)\}\cap \bigcap_{k=1}^{N/2}\{D'(k)\le (1+\varepsilon)a_{n,-b}\psi(n)\},
\]
which implies 
\begin{align}
    E\Big(\sum_*\!^2\Big)&=E\Big(\sum_*\!^2 \mathrel{;}F\Big)+E\Big(\sum_*\!^2 \mathrel{;}F^c\Big)   \notag\\
        &\le (1+\varepsilon)^2a_{n,-b}^2\psi(n)^2E\left(\left\{2\sum_{j=1}^{N/2}Z_i\right\}^2\right)+n^2P(F^c)  \notag\\
        &\le (1+\varepsilon)^2a_{n,-b}^2\psi(n)^2NE\left(2\sum_{j=1}^{N/2}Z_i\right)+E(D_n)^2\log nP(F^c)  \notag \\
        &\le (1+\varepsilon)^2a_{n,-b}^2\psi(n)^2N^2P(Z_i=1)+E(D_n)^2\log nP(F^c)  \notag \\
        &\le E(D_n)^2\left((1+\varepsilon)^2\frac{\log\log n}{\log n}+\log nP(F^c)\right).  \label{lemDn-third-P(F^c)}
\end{align}

Now it remains to give an upper bound of $P(F^c)$. 
We define 
\begin{gather*}
    d(i)=d_{\mathcal{G}_{(i-1)a_{n,-3},ia_{n,-3}}}(S((i-1)a_{n,-3}),S(ia_{n,-3}))\mbox{ for }i=1,2,\cdots, (\log n)^3, \\
    \sum_{\star\star}=\sum_{i=1}^{(\log n)^3}d(i).
\end{gather*}
The same argument as \eqref{var-sumstar} yields 
\[
    \mathrm{Var}\Big(\sum_{\star\star}\Big)\le cE(D_n)^2(\log n)^{-5/2}.
\]
Moreover, we have 
\[
    E\Big(\sum_{\star\star}\Big)=(\log n)^3a_{n,-3}\psi(a_{n,-3})\le \left(1+\frac{\varepsilon}{2}\right)E(D_n),
\]
where the last inequality follows from \eqref{eqiv1}. 
Thus, by Markov's inequality, we have that 
\begin{align*}
    P(D_n\ge (1+\varepsilon)E(D_n))&\le P\Big(\sum_{\star\star}\ge (1+\varepsilon)E(D_n)\Big)    \\
        & \le P\Big(\Big|\sum_{\star\star}-E\Big(\sum_{\star\star}\Big)\Big|\ge \frac{\varepsilon}{2}E(D_n)\Big)    \\
        &\le 4\varepsilon^{-2}E(D_n)^{-2}\mathrm{Var}\Big(\sum_{\star\star}\Big)    \\
        &\le c_\varepsilon (\log n)^{-5/2},
\end{align*}
where $c_\varepsilon$ is a constant depending only on $\varepsilon$. 
By taking the union bound, we have that 
\[
    P(F^c)\le cN(\log n)^{-5/2}\le c(\log n)^{-17/12}. 
\]
Substituting this into \eqref{lemDn-third-P(F^c)} and combining with \eqref{lemDn-third-term}, we obtain 
\[
    E\left(\Big(D_n-\sum_\star\Big)^2\right)\le cE(D_n)^2(\log n)^{-5/12}.
\]
Finally, combining this with \eqref{lem-var-sum},\eqref{lem-var-1st} and \eqref{var-sumstar} yields \eqref{lem-var}.
\end{proof}

Using Lemma \ref{var}, we obtain the following lemma, which establishes a sufficiently fast decay of the upper tail of the graph distance between any two points relative to $E(D_{n})$.

\begin{lemma}\label{uppertail}
Let $\epsilon \in (0,\frac{1}{2})$. Then we have 
$$P \Big( \max \big\{ d_{{\cal G}_{t,t'}} \big( S(t), S(t') \big) \ \big| \ 0 \le t \le t' \le n \big\} \ge E (D_{n} ) \big\{ 1 + (\log n)^{-\epsilon} \big\} \Big) \le C (\log n)^{-\frac{13}{4} + 2 \epsilon}.$$
In particular, 
$$P \Big( D_{n} \ge E (D_{n} ) \big\{ 1 + (\log n)^{-\epsilon} \big\} \Big) \le C (\log n)^{-\frac{13}{4} + 2 \epsilon}.$$
\end{lemma}

\begin{proof}
Recall the definition of $\displaystyle \sum_\star$ (see \eqref{lem-def-sumstar} for this). The triangle inequality ensures $\displaystyle \sum_\star \ge D_n $.

Therefore, if $\epsilon \in (0,\frac{1}{2})$, we have 
\begin{align*}
&P \Big( D_{n} \ge E (D_{n} ) \big\{ 1 + (\log n)^{-\epsilon} \big\} \Big) \le P \Big( \sum_{\star} \ge E (D_{n} ) \big\{ 1 + (\log n)^{-\epsilon} \big\} \Big) \\
&\le P \Big( \sum_{\star} \ge E \big( \sum_{\star} \big) \big\{ 1 + c (\log n)^{-\epsilon} \big\} \Big) \le C \text{Var} \big(  \sum_{\star} \big)  E \big( \sum_{\star} \big)^{-2} (\log n)^{ 2 \epsilon}  \\
&\le C (\log n)^{-\frac{13}{4} + 2 \epsilon},
\end{align*}
where 
we used 

\begin{itemize}
\item  \eqref{dif} and the fact that $\epsilon \in (0,\frac{1}{2})$ and $E(D_{n}) \ge c n (\log n)^{-\frac{1}{2}} $ (see \eqref{LCTN} and \eqref{LCTN1}) to replace $E (D_{n} )$ with $E \big(\displaystyle \sum_{\star} \big)$ in the second inequality,

\item Lemma \ref{var} in the last inequality.
    
\end{itemize}
This gives the second assertion. 

As for the first assertion, note that if $0 \le t \le t' \le n$,
$$ d_{{\cal G}_{t,t'}} \big( S(t), S(t') \big) \le 2 a_{n,-3} + \sum_{i=1}^{M} d_{{\cal G}_{(i-1)a_{n, -3}, i a_{n, -3}}} \Big( S \big( (i-1) a_{n,-3} \big), S \big( i a_{n,-3} \big) \Big).$$
Since $E(D_{n}) \ge c n (\log n)^{-\frac{1}{2}}$ and $\epsilon \in (0,1/2)$, by repeating the same argument as above, we get the first assertion.  
\end{proof}

\begin{remark}\label{slow}
Let $r \in (0,4)$. By exactly the same argument as that used to prove \eqref{eqiv1}, it follows that for $k \in [n/2, n]$
\begin{equation}\label{eqiv2}
\psi (k) =\psi (n) \Big\{ 1 + O \Big( \frac{\log \log n}{\sqrt{\log n}} \Big) \Big\}.
\end{equation}
Thus, our $\psi$ defined in \eqref{psi} satisfies the condition (I) in  Lemma \ref{deriv}.
\end{remark}

Let us now 
verify the condition (II) in  Lemma \ref{deriv}. The next proposition gives an upper bound of this expectation $E (D_{n}^{1} + D_{n}^{2}-D_{n})$, \textit{i.e.}\ the left-hand side of \eqref{eq:Dn-s}.

\begin{proposition}\label{upper}
We have 
$$\psi(n/2) - \psi (n) \le  \frac{\log 2}{2 \log n}  \psi (n/2)  \big\{ 1 + O \big( (\log n)^{-1/100} \big) \big\}.$$
\end{proposition}

\begin{proof}
Since \cite{exact} considers only a lower bound for $E (D_{n}^{1} + D_{n}^{2}-D_{n})$, we provide the details of the proof here.

Recall that $b= \frac{13}{12}$.  Let $V_{1}$ denote the event that the following three conditions are fulfilled:
\begin{itemize}
\item $D_{n}^{1} \vee D_{n}^{2} \vee D_{n} \le 2n \psi (n)$,

\item $D(j) \le n (\log n)^{-b} \psi (n)   \big\{ 1 + O \big( (\log n)^{-1/12} \big) \big\}$ for all $j =1,2, \cdots , N/2$,\hspace{\fill}\inlinelabel{condition-2}

\item $D(k)' \le n (\log n)^{-b} \psi (n)  \big\{ 1 + O \big( (\log n)^{-1/12} \big) \big\}$ for all $k = 1,2, \cdots , N/2$.\hspace{\fill}\inlinelabel{condition-3}

\end{itemize}
We recall that the notation used here is defined immediately after Remark \ref{prev}.
Lemma \ref{uppertail} and \eqref{eqiv2} guarantee that 
$$P (V_{1}) = 1 - O \big( (\log n)^{-2} \big).$$
Therefore, 
\begin{equation}\label{1st}
E (D_{n}^{1} + D_{n}^{2}-D_{n}) = E (D_{n}^{1} + D_{n}^{2}-D_{n} \ ; \ V_{1}) + O \big( n (\log n)^{-2} \big).
\end{equation}
Since $\psi (n) \ge c (\log n)^{-1/2}$ (see \eqref{LCTN} and \eqref{LCTN1} for this), we note that in order to check the condition (II) in Lemma \ref{deriv}, every error term must be $O \big( n (\log n)^{-\frac{3}{2} - \epsilon} \big)$ for some $\epsilon > 0$. 

Next, we consider $L$, defined in \eqref{L}, the size of the longest intersection. Distinguishing cases according to the possible values of $L$, 
\begin{align}
&E (D_{n}^{1} + D_{n}^{2}-D_{n} \ ; \ V_{1}) = E (D_{n}^{1} + D_{n}^{2}-D_{n} \ ; \ V_{1}, \, L=2) \notag \\
&+ \sum_{l=3}^{N/2} E (D_{n}^{1} + D_{n}^{2}-D_{n} \ ; \ V_{1}, \, L=l) + \sum_{l=N/2 +1}^{N} E (D_{n}^{1} + D_{n}^{2}-D_{n} \ ; \ V_{1}, \, L=l). \label{2nd}
\end{align}

We first deal with the second term. Let $3 \le l \le N/2$. If $L=l$, there exists $j \in \{ 1,2, \cdots , l-1 \}$ such that $I_{j} \leftrightarrow I_{l-j}'$ occurs. Therefore, 
$$E (D_{n}^{1} + D_{n}^{2}-D_{n} \ ; \ V_{1}, \, L=l) \le \sum_{j=1}^{l-1} E \big( D_{n}^{1} + D_{n}^{2}-D_{n} \ ; \ V_{1}, \, I_{j} \leftrightarrow I_{l-j}' \big).$$
With this in mind, let 
$$V_{2} = V_{2; l, j} := V_{1} \cap \{  I_{j} \leftrightarrow I_{l-j}' \}.$$

We aim to show that given $\{  I_{j} \leftrightarrow I_{l-j}' \}$, with high (conditional) probability, there exist $T \in I_{j+1}$ and $T' \in I_{l-j+1}'$ such that  $S[0,T] \cap S[T+1, n] = S[0,T'] \cap S[T'+1, n]= \emptyset$ (recall that such $T$ and $T'$ were called cut times up to time $n$). Let 
$$V' = \big\{ \exists t \in [t_{j+1}, t_{j+1} + a_{n,-6}] \text{ such that } S[0, t] \cap S[t+1, t_{j}] = \emptyset \big\}.$$
Note that $V'$ and $\{  I_{j} \leftrightarrow I_{l-j}' \}$ are independent.  \cite[Lemma 7.7.4]{Lawb} implies that 
$$P (V') = 1 -O \Big( \frac{\log \log n}{\log n} \Big).$$
Suppose that both $V'$ and $\{  I_{j} \leftrightarrow I_{l-j}' \}$ occur. Take $T  \in [t_{j+1}, t_{j+1} + a_{n,-6}] \text{ such that } S[0, T] \cap S[T+1, t_{j}] = \emptyset$.

To achieve the aforementioned aim, we will prove 
that in fact $T$ becomes a cut time up to time $n$ with high probability. If $T$ is not a cut time up to  $n$, the following intersection must occur: 
$$S[0, t^{\ast}] \cap S[t_{j}, n] \neq \emptyset,$$ 
where $t^{\ast} := (t_{j+1}+t_{j})/2$. Note that this 
event is not independent of the event $I_{j} \leftrightarrow I_{l-j}'$. 
However, we can see that these two events are 
almost independent in the following sense, resulting in \eqref{deco} below. By translating $S[0,n]$ so that $S(t_{j})$ is at the origin, and then time-reversing the path $S[0, t_{j}]$, we note that 
\begin{align}
&P \Big( S[0, t^{\ast}] \cap S[t_{j}, n] \neq \emptyset, \, I_{j} \leftrightarrow I_{l-j}'  \Big) \notag \\
& \le P \Big( S^{1}[a_{n,-b}/2, (N/2-j)a_{n,-b}] \cap S^{2}[0, n] \neq  \emptyset, \notag \\
& \ \ \  \ \ \  \  \ \ \ S^{2}[0, a_{n,-b}] \cap S^{2} [(l-1)a_{n,-b}, l a_{n,-b}] \neq \emptyset \Big), \label{FAP}
\end{align}
where $S^{1}$ and $S^{2}$ are independent random walks started at the origin (we write $P_{i}^{x}$ for the probability law of $S^{i}$ assuming $S^{i}(0) =x$). The time-reversal of $S[0, t_{j}]$ corresponds to $S^{1}$, while $S[t_{j}, n]$ corresponds to $S^{2}$. To bound the right-hand side of \eqref{FAP} from above, take $q \ge1$ and let 
\begin{equation}\label{YQ}
Y_{q} := \max \Big\{ P_{1}^{x} \big( S^{1} [0, n] \cap S^{2} [0, n] \neq \emptyset \big) \ \Big| \ x \in \mathbb{Z}^{4}, \, \text{dist} \big( x, S^{2} [0, n] \big) \ge \sqrt{n} (\log n)^{-q} \Big\}.
\end{equation}
Here, $\mathrm{dist}(A,B)$ denotes the distance between sets $A$ and $B$ with respect to the Euclidean metric on $\mathbb{R}^4$. Note that $Y_{q}$ is a random variable measurable with respect to $S^{2} [0, n]$. Combining \cite[Propositions 4.1]{slowly} and \cite[Propositions 4.3]{slowly} ensures that for any $p, q \ge 1$, there exists $C >0$ such that 
\begin{equation}\label{freez}
P_{2}^{0} \Big( Y_{q} \le C \frac{(\log \log n)^{C}}{\log n} \Big) \ge 1 - C (\log n)^{-p}.
\end{equation}
Here we note that, as stated immediately before \cite[Proposition 4.1]{slowly}, the powers of the logarithm (that is, $p$ and $q$) can be chosen arbitrarily large, by taking the constant $C$ sufficiently large depending on $p$ and $q$.
This type of estimate is sometimes called a `freezing lemma'.
 By \cite[Theorem 1.2.1]{Lawb} and  \cite[Proposition 1.5.10]{Lawb}, it is easy to see that 
\begin{equation}\label{freez2}
P \Big( \text{dist} \big( S^{1} (a_{n,-b}/2), S^{2} [0, n] \big) \le \sqrt{n} (\log n)^{-20} \Big) \le C (\log n)^{-10}.
\end{equation}
With this in mind, we use \eqref{freez} for $p=10$ and $q=20$. Proposition \ref{fnk} implies that 
\begin{equation}\label{FR3}
P( I_{j} \leftrightarrow I_{l-j}' ) \asymp (\log n)^{-1} l^{-2} \ge (\log n)^{-1} N^{-2} = (\log n)^{-\frac{19}{6}},
\end{equation}
which is much bigger than the error terms that appear in \eqref{freez} and \eqref{freez2}. We now apply the freezing lemma \eqref{freez} to estimate the probability appearing on the right-hand side of \eqref{FAP}. For this purpose, we proceed in the following steps. (In what follows, we will apply the freezing lemma \eqref{freez} several times, following the same procedure.)

\begin{itemize}
\item Let $p=10$ and $q=20$. Choosing $C>0$ so that \eqref{freez} holds for this choice of $(p,q)$, we set 
$$A_{\ast} = \bigg\{ Y_{q} \le C \frac{(\log \log n)^{C}}{\log n} \bigg\}. $$

\item Also, we write
$$A_{\ast \ast} = \Big\{ \text{dist} \big( S^{1} (a_{n,-b}/2), S^{2} [0, n] \big) \ge \sqrt{n} (\log n)^{-20} \Big\}. $$
By \eqref{freez} and \eqref{freez2}, 
\begin{align}
&P \Big( S^{1}[a_{n,-b}/2, (N/2-j)a_{n,-b}] \cap S^{2}[0, n] \neq  \emptyset, \notag \\
& \ \ \ \ \ \ \ S^{2}[0, a_{n,-b}] \cap S^{2} [(l-1)a_{n,-b}, l a_{n,-b}] \neq \emptyset \Big) \notag \\
&\le P \Big( S^{1}[a_{n,-b}/2, (N/2-j)a_{n,-b}] \cap S^{2}[0, n] \neq  \emptyset, \notag \\
& \ \ \ \ \ \ \ S^{2}[0, a_{n,-b}] \cap S^{2} [(l-1)a_{n,-b}, l a_{n,-b}] \neq \emptyset, \, A_{\ast} \cap A_{\ast \ast} \Big) +  O \big( (\log n)^{-10} \big).\label{ineq:non-intersec} 
\end{align}

\item In order to estimate the probability appearing on the right-hand side of \eqref{ineq:non-intersec}, we use the Markov property for $S^{1}$ at time $a_{n,-b}/2$. Since on the event $A_{\ast} \cap A_{\ast \ast}$, the probability that a random walk started at $S^{1} (a_{n,-b}/2)$ hits $S^{2} [0, n]$ by time $n$ is less than $C (\log n)^{-1} (\log \log n)^{C}$, 
the right-hand side of \eqref{ineq:non-intersec} is bounded above by the following quantity.
\begin{align}
&C (\log n)^{-1} (\log \log n)^{C} P \Big( S^{2}[0, a_{n,-b}] \cap S^{2} [(l-1)a_{n,-b}, l a_{n,-b}] \neq \emptyset \Big) \notag \\
&\ \ \ \ \  + O \big( (\log n)^{-10} \big) \notag \\
&\le C (\log n)^{-1} (\log \log n)^{C} P( I_{j} \leftrightarrow I_{l-j}' ) + O \big( (\log n)^{-10} \big) \notag \\
&\le C (\log n)^{-1} (\log \log n)^{C} P( I_{j} \leftrightarrow I_{l-j}' ), \label{deco}
\end{align}
where we used \eqref{FR3} in the last inequality.

\end{itemize}

Therefore, if  
\begin{align*}
    V''&=\left\{\begin{gathered} \exists T \in [t_{j+1}, t_{j+1} + a_{n,-6}] \text{ and } \exists T' \in [t_{l-j+1}'-a_{n,-6}, t_{l-j+1}']  \text{ such that } \\
S[0, T] \cap S[T+1, n] =S[0, T'] \cap S[T'+1, n] = \emptyset 
    \end{gathered}\right\},   \\
V_{3} &= V_{2} \, \cap \, V'',
\end{align*}
we have 
\begin{align}
&E \big( D_{n}^{1} + D_{n}^{2}-D_{n} \ ; \ V_{2} \big) = E \big( D_{n}^{1} + D_{n}^{2}-D_{n} \ ; \ V_{3} \big) + E \big( D_{n}^{1} + D_{n}^{2}-D_{n} \ ; \ V_{2} \cap (V'')^{c} \big) \notag \\
&\le E \big( D_{n}^{1} + D_{n}^{2}-D_{n} \ ; \ V_{3} \big) + 4 n \psi (n) P \big( V_{2} \cap (V'')^{c} \big) \notag \\
&\le E \big( D_{n}^{1} + D_{n}^{2}-D_{n} \ ; \ V_{3} \big) + C n \psi (n) (\log n)^{-1}(\log \log n)^{C} P \big(  I_{j} \leftrightarrow I_{l-j}' \big).\notag 
\end{align}

Suppose that $V_{3}$ occurs. Take two cut times $T$ and $T'$ up to $n$ claimed in $V''$. Note that 
\begin{align*}
D_{n}^{1} + D_{n}^{2} - D_{n} &\le d_{{\cal G}_{T, n/2}} \big( S(T), S(n/2) \big) + d_{{\cal G}_{n/2, T'}} \big( S(n/2), S(T') \big) \\
&\le d_{{\cal G}_{t_{j+1}, n/2}} \big( S(t_{j+1}), S(n/2) \big) + d_{{\cal G}_{n/2, t_{l-j+1}'}} \big( S(n/2), S(t_{l-j+1}') \big).
\end{align*}
The 
conditions \eqref{condition-2} and \eqref{condition-3} of $V_{1}$ ensure that the right-hand side above is smaller than
$$(l+2) n (\log n)^{-b} \psi (n) \big\{ 1 + O \big( (\log n)^{-1/12} \big) \big\}.$$
Thus, we have 
\begin{align*}
&E \big( D_{n}^{1} + D_{n}^{2}-D_{n} \ ; \ V_{2} \big) \\
&\le E \big( D_{n}^{1} + D_{n}^{2}-D_{n} \ ; \ V_{3} \big) + C n \psi (n) (\log n)^{-1}(\log \log n)^{C} P \big(  I_{j} \leftrightarrow I_{l-j}' \big) \\
&\le (l+2) n (\log n)^{-b} \psi (n) \big\{ 1 + O \big( (\log n)^{-1/12} \big) \big\} P (V_{3}) \\
& \ \ \ \ \ \ \ \ + C n \psi (n) (\log n)^{-1}(\log \log n)^{C} P \big(  I_{j} \leftrightarrow I_{l-j}' \big) \\
&\le (l+2) n (\log n)^{-b} \psi (n) \big\{ 1 + O \big( (\log n)^{-1/12} \big) \big\} P \big(  I_{j} \leftrightarrow I_{l-j}' \big)\\
& \ \ \ \ \ \ \ \ + C n \psi (n) (\log n)^{-1}(\log \log n)^{C} P \big(  I_{j} \leftrightarrow I_{l-j}' \big).
\end{align*}
Consequently, for $3 \le l \le N/2$, we have 
\begin{align*}
&E (D_{n}^{1} + D_{n}^{2}-D_{n} \ ; \ V_{1}, \, L=l) \le \sum_{j=1}^{l-1} E \big( D_{n}^{1} + D_{n}^{2}-D_{n} \ ; \ V_{1}, \, I_{j} \leftrightarrow I_{l-j}' \big) \\
&\le \sum_{j=1}^{l-1} (l+2) n (\log n)^{-b} \psi (n) \big\{ 1 + O \big( (\log n)^{-1/12} \big) \big\} P \big(  I_{j} \leftrightarrow I_{l-j}' \big) \\
& \ \ \ \ + \sum_{j=1}^{l-1} C n \psi (n) (\log n)^{-1}(\log \log n)^{C} P \big(  I_{j} \leftrightarrow I_{l-j}' \big) \\
&\le l^{2} n (\log n)^{-b} \psi (n) \big\{ 1 + O \big( (\log n)^{-1/12} \big) \big\} P \big(  I_{1} \leftrightarrow I_{l-1}' \big) \\
& \ \ \ \ + C l n \psi (n) (\log n)^{-1}(\log \log n)^{C} P \big(  I_{1} \leftrightarrow I_{l-1}' \big).
\end{align*}

We can handle the case 
$\frac{N}{2} + 1 \le l \le N$ with the exactly same argument as in the previous case. 
If $L=l$, there exists $j \in \{ l- \frac{N}{2}, \cdots, \frac{N}{2}  \}$ such that $I_{j} \leftrightarrow I_{l-j}'$ occurs (so there are $N-l+1$ choices for $j$) 
and we can conclude that for $\frac{N}{2} + 1 \le l \le N$, 
\begin{align*}
&E (D_{n}^{1} + D_{n}^{2}-D_{n} \ ; \ V_{1}, \, L=l)  \\
&\le l (N-l) n (\log n)^{-b} \psi (n) \big\{ 1 + O \big( (\log n)^{-1/12} \big) \big\} P \big(  I_{N/2} \leftrightarrow I_{l-N/2}' \big) \\
&\ \ \ \ + C (N-l) n \psi (n) (\log n)^{-1}(\log \log n)^{C} P \big(  I_{N/2} \leftrightarrow I_{l-N/2}' \big).
\end{align*}
(An easy adaptation to derive this is left to the reader.) Therefore, summing for $l=3,4, \cdots , N$, we have 
\begin{align}
&\sum_{l=3}^{N/2} E (D_{n}^{1} + D_{n}^{2}-D_{n} \ ; \ V_{1}, \, L=l) + \sum_{l=N/2 +1}^{N} E (D_{n}^{1} + D_{n}^{2}-D_{n} \ ; \ V_{1}, \, L=l) \notag \\
&\le \sum_{l=3}^{N/2} l^{2} n (\log n)^{-b} \psi (n) \big\{ 1 + O \big( (\log n)^{-1/12} \big) \big\} P \big(  I_{1} \leftrightarrow I_{l-1}' \big) \notag \\
& \ \ \  \ + C \sum_{l=3}^{N/2}  l n \psi (n) (\log n)^{-1}(\log \log n)^{C} P \big(  I_{1} \leftrightarrow I_{l-1}' \big) \notag \\
&\ \ \ \ + \sum_{l=N/2 +1}^{N} l (N-l) n (\log n)^{-b} \psi (n) \big\{ 1 + O \big( (\log n)^{-1/12} \big) \big\} P \big(  I_{N/2} \leftrightarrow I_{l-N/2}' \big) \notag  \\
&\ \ \ \  + C \sum_{l=N/2 +1}^{N} (N-l) n \psi (n) (\log n)^{-1}(\log \log n)^{C} P \big(  I_{N/2} \leftrightarrow I_{l-N/2}' \big). \label{comp}
\end{align}
On the other hand, Proposition \ref{fnk} guarantees that 
\begin{align*}
&P \big(  I_{1} \leftrightarrow I_{l-1}' \big) = \frac{1}{2 \log n} \log \Big[ 1 + \frac{1}{l^{2}-1} \Big] \Big\{ 1 + O \Big( (\log n)^{-\frac{3}{25}} \Big) \Big\}, \\
&P \big(  I_{N/2} \leftrightarrow I_{l-N/2}' \big) = \frac{1}{2 \log n} \log \Big[ 1 + \frac{1}{l^{2}-1} \Big] \Big\{ 1 + O \Big( (\log n)^{-\frac{3}{25}} \Big) \Big\}. 
\end{align*}
Substituting this and carrying out a careful calculation (it is here that we essentially use the sharp estimate obtained in Proposition \ref{fnk}), we see that 
\begin{align*}
&\mathrm{(i)} \  C \sum_{l=3}^{N/2}  l n \psi (n) (\log n)^{-1}(\log \log n)^{C} P \big(  I_{1} \leftrightarrow I_{l-1}' \big) = O \Big( n \psi (n) (\log n)^{-2} (\log \log n)^{C} \Big), \\
&\mathrm{(ii)} \  C \sum_{l=N/2 +1}^{N} (N-l) n \psi (n) (\log n)^{-1}(\log \log n)^{C} P \big(  I_{N/2} \leftrightarrow I_{l-N/2}' \big)  \\
& \  \ \ \ \ \ \ \ = O \Big( n \psi (n) (\log n)^{-2} (\log \log n)^{C} \Big), \\
&\mathrm{(iii)} \ \sum_{l=3}^{N/2} l^{2} n (\log n)^{-b} \psi (n) \big\{ 1 + O \big( (\log n)^{-1/12} \big) \big\} P \big(  I_{1} \leftrightarrow I_{l-1}' \big) \\
&\ \ \ \ \ \   - \sum_{l=N/2 +1}^{N} l^{2} n (\log n)^{-b} \psi (n) \big\{ 1 + O \big( (\log n)^{-1/12} \big) \big\} P \big(  I_{N/2} \leftrightarrow I_{l-N/2}' \big) \\ 
& \ \ \ \ \ \ \ = O \Big( n \psi (n) (\log n)^{-\frac{13}{12}} \Big), \\
&\mathrm{(iv)} \ \sum_{l=N/2 +1}^{N} l N n (\log n)^{-b} \psi (n) \big\{ 1 + O \big( (\log n)^{-1/12} \big) \big\} P \big(  I_{N/2} \leftrightarrow I_{l-N/2}' \big)  \\
& \ \ \ \ \ \ \ \ \  = \frac{\log 2}{2 \log n} n \psi (n)  \big\{ 1 + O \big( (\log n)^{-1/12} \big) \big\}.
\end{align*}
Note that in (iii) it is important to consider the difference between the two sums. We also note that the leading term arises from the sum in (iv).

The remainder concerns the case $L = 2$. In this case, by definition of $L$,
\begin{equation}\label{constr}
S[0, t_1] \cap S[n/2, n] = S[0, n/2] \cap S[t_{1}', n] = \emptyset.
\end{equation}
We aim to show that there exist small $j$ and  $t \in [t_{j}, n/2]$ such that $t$ is a cut time up to $n$. 
If we define 
\begin{align*}
&H_{1} := \Big\{ \exists T_{1} \in [t_{2}, t_{2}+a_{n,-6}] \text{ such that } S[t_{3}, T_{1}] \cap S[T_{1}+1, n/2] = \emptyset \Big\} \text{ and } \\
&H_{2} := \Big\{ \exists T_{2} \in [t_{4}-a_{n,-6}, t_{4}] \text{ such that } S[0, T_{2}] \cap S[T_{2}+1, t_{3}] = \emptyset \Big\},
\end{align*}
then $H_{1}$ and $H_{2}$ are independent and $P (H_{i}) = 1 - O \big( (\log n)^{-1} \log \log n \big)$ for $i=1,2$.  Thus, 
\begin{equation}\label{hh}
E (D_{n}^{1}+ D_{n}^{2} - D_{n} \ ; \ V_{1}, \, L=2, \, H_{1}^{c} \cap H_{2}^{c}) = O \Big( n \psi (n) (\log n)^{-2} (\log \log n)^{2} \Big).\end{equation}
Suppose that $H_{1}$ occurs but $T_{1}$ is not a cut time up to $n$. Then the constraint of \eqref{constr} guarantees that $S[0, t_{3}] \cap S[t_{2}, n/2] \neq \emptyset$. Let $j_{n} = (\log n)^{\frac{1}{100}}$. By Proposition \ref{fnk}, we know that 
$$P \Big( S[0, t_{j_{n}}] \cap S[t_{2}, n/2] \neq \emptyset \Big) \le C \sum_{j=j_{n}}^{N/2} (\log n)^{-1} j^{-2} \le C (\log n)^{-\frac{101}{100}}.$$
Since $-\frac{101}{100} < -1$, we may regard this as a negligible event. So, suppose that $S[t_{j_{n}}, t_{3}] \cap S[t_{2}, n/2] \neq \emptyset$, and let 
$$H_{1}' = \{ \exists T_{1}' \in [t_{j_{n}+1}-a_{n,-6}, t_{j_{n}+1}] \text{ such that } T_{1}' \text{ is a cut time up to } n \}.$$
A similar method to that used to prove \eqref{deco} shows that 
\begin{align*}
P \Big( S[t_{j_{n}}, t_{3}] \cap S[t_{2}, n/2] \neq \emptyset, \, (H_{1}')^{c} \Big) &\le C \frac{(\log \log n)^{C}}{\log n} P \big(S[t_{j_{n}}, t_{3}] \cap S[t_{2}, n/2] \neq \emptyset \big) \\
&\le C \frac{(\log \log n)^{C}}{(\log n)^{2}},
\end{align*}
where we used Proposition \ref{fnk} in the last inequality. Consequently,
\begin{align*}
&E (D_{n}^{1}+ D_{n}^{2} - D_{n} \ ; \ V_{1}, \, L=2, \, H_{1}) \\
&\le E (D_{n}^{1}+ D_{n}^{2} - D_{n} \ ; \ V_{1}, \, L=2, \, H_{1}, \, T_{1} \text{ is a cut time up to } n ) \\
& + E \Big( D_{n}^{1}+ D_{n}^{2} - D_{n} \ ; \ V_{1}, \, L=2, \, H_{1}, \, S[0, t_{3}] \cap S[t_{2}, n/2] \neq \emptyset  \Big) \\
&\le E (D_{n}^{1}+ D_{n}^{2} - D_{n} \ ; \ V_{1}, \, L=2, \, H_{1}, \, T_{1} \text{ is a cut time up to } n ) \\
&+ E \Big( D_{n}^{1}+ D_{n}^{2} - D_{n} \ ; \ V_{1}, \, L=2, \, H_{1}, \, S[t_{j_{n}}, t_{3}] \cap S[t_{2}, n/2] \neq \emptyset  \Big) \\ 
&\ \ \ \ \ \ + O \Big( n \psi (n) (\log n)^{-\frac{101}{100}} \Big) \\
&\le E (D_{n}^{1}+ D_{n}^{2} - D_{n} \ ; \ V_{1}, \, L=2, \, H_{1}, \, T_{1} \text{ is a cut time up to } n ) \\
&+ E \Big( D_{n}^{1}+ D_{n}^{2} - D_{n} \ ; \ V_{1}, \, L=2, \, H_{1}, \, S[t_{j_{n}}, t_{3}] \cap S[t_{2}, n/2] \neq \emptyset, \, H_{1}'  \Big) \\
& \ \ \ \ \ \ + O \Big( n \psi (n) (\log n)^{-\frac{101}{100}} \Big) \\
&\le E \Big( D_{n}^{1}+ D_{n}^{2} - D_{n} \ ; \ V_{1}, \, L=2, \,  \exists T_{1} \in [t_{2}, t_{2}+a_{n,-6}] \\
&\ \ \ \ \ \ \ \ \ \ \ \text{ such that $T_{1}$ a cut time up to } n \Big) \\
&+ E \Big( D_{n}^{1}+ D_{n}^{2} - D_{n} \ ; \ V_{1}, \, L=2, \,  \exists T_{1}' \in [t_{j_{n}+1}-a_{n,-6}, t_{j_{n}+1}] \\
&  \ \ \ \ \ \ \ \ \ \ \ \text{ such that } T_{1}' \text{ is a cut time up to } n \Big) \\
&\ \ \ \ \ \ + O \Big( n \psi (n) (\log n)^{-\frac{101}{100}} \Big) \\
&\le E \Big( D_{n}^{1}+ D_{n}^{2} - D_{n} \ ; \ V_{1}, \, L=2, \,  \exists T_{1} \in [t_{j_{n}+2}, n/2] \\
&\ \ \ \ \ \ \ \ \ \ \ \ \ \text{ such that } T_{1} \text{ is a cut time up to } n \Big) \\
&\ \ \ \ \ \ + O \Big( n \psi (n) (\log n)^{-\frac{101}{100}} \Big).
\end{align*}
Recalling \eqref{hh}, the same procedure for $H_{2}$ (and for $[n/2, n]$) yields that 
\begin{align*}
&E (D_{n}^{1}+ D_{n}^{2} - D_{n} \ ; \ V_{1}, \, L=2) \\
&\le E \Big( D_{n}^{1}+ D_{n}^{2} - D_{n} \ ; \ V_{1}, \, L=2, \,  \exists T_{1} \in [t_{j_{n}+2}, n/2]  \text{ and } T_{2} \in [n/2, t_{j_{n}+2}']  \\ 
& \ \ \ \ \ \ \ \ \ \text{ such that } T_{1} \text{ and } T_{2} \text{ are cut times up to } n \Big) + O \Big( n \psi (n) (\log n)^{-\frac{101}{100}} \Big) \\
&\le 2 (j_{n}+2) n (\log n)^{-b} \psi (n) + O \Big( n \psi (n) (\log n)^{-\frac{101}{100}} \Big) = O \Big( n \psi (n) (\log n)^{-\frac{101}{100}} \Big).
\end{align*}
Combining this estimate with (i)-(iv) above, \eqref{1st}, \eqref{2nd} and \eqref{comp}, we have 
$$E (D_{n}^{1}+ D_{n}^{2} - D_{n} ) \le \frac{\log 2}{2 \log n} n \psi (n)  \big\{ 1 + O \big( (\log n)^{-1/100} \big) \big\}.$$
Dividing both sides by $n$ and using \eqref{eqiv2}, 
\begin{align*}
\psi(n/2) - \psi (n) &\le \frac{\log 2}{2 \log n}  \psi (n)  \big\{ 1 + O \big( (\log n)^{-1/100} \big) \big\} \\
&= \frac{\log 2}{2 \log n}  \psi (n/2)  \big\{ 1 + O \big( (\log n)^{-1/100} \big) \big\}, 
\end{align*}
as desired.
\end{proof}

 The next proposition gives a lower bound of $\psi(n/2) - \psi (n)$ via a more detailed estimate of a lower bound of $E (D_{n}^{1}+ D_{n}^{2} - D_{n})$ (recall \eqref{eq:Dn-s}).

\begin{proposition}\label{lower}
We have 
$$\psi(n/2) - \psi (n) \ge  \frac{\log 2}{2 \log n}  \psi (n/2)  \big\{ 1 + O \big( (\log n)^{-c} \big) \big\}.$$
\end{proposition}

Combining this with \eqref{eqiv2} and Proposition \ref{upper}, we see that our $\psi$ defined in \eqref{psi} satisfies the conditions (I) and (II) in Lemma \ref{deriv}. 

\begin{proof}
We will use the same notation introduced at the beginning of this section. The argumentation in this proof is very similar to that of   \cite[Theorem 1.2.3]{exact}. An important difference from \cite{exact} is that here we need to give all error terms carefully.

 Recall that $L$ is defined in \eqref{L} and $b = \frac{13}{12}$. As in the proof of  \cite[Theorem 1.2.3]{exact}, let us introduce the following events $W^{i}$ (the event $W^{i}$ here corresponds to $A^{i}$ in the proof of \cite[Theorem 1.2.3]{exact}): 
 \begin{align*}
&W^{1} =\{ I_{j} \leftrightarrow I_{l-j}' \}, \\
&W^{2} = \{ L= l \}, \\
&W^{3} = \Big\{  \text{there exists just  one pair } (j,k) \text{ which attains the maximum of } L  \Big\}, \\
&W^{4} = \Big\{ \exists T \in I_{j+1} \text{ such that } T \text{ is a cut time up to } n \Big\}, \\
&W^{5} = \Big\{ \exists T' \in I_{l-j+1}' \text{ such that } T' \text{ is a cut time up to } n \Big\}, \\
&W^{6} = \Big\{ \max \big\{  d_{{\cal G}_{t, t'}} \big( S(t), S(t') \big) \  \big| \ 0 \le t'-t \le a_{n, -b}, \, t \in I_{j+1} \cup I_{j} \cup I_{j-1} \big\} \le 2 a_{n,-b} \psi (n)  \Big\}, \\
&W^{7} = \Big\{ \max \big\{  d_{{\cal G}_{t, t'}} \big( S(t), S(t') \big) \  \big| \ 0 \le t'-t \le a_{n, -b}, \, t \in I_{l-j-1}' \cup I_{l-j}' \cup I_{l-j+1}' \big\} \\ 
&  \ \ \ \ \ \ \ \ \  \le 2 a_{n,-b} \psi (n)  \Big\}, \\
&W^{8} = \Big\{ \exists U \in I_{j-1} \text{ such that } U \text{ is a cut time up to } n/2 \Big\}, \\
&W^{9} = \Big\{ \exists U' \in I_{l-j-1}' \text{ such that } S[n/2, U'] \cap S[U'+1, n] = \emptyset \Big\}, \\
&W^{10} = \Big\{ d_{{\cal G}_{t_{j}, n/2}} \big( S (t_{j}), S(n/2) \big) \ge j a_{n, -b} \psi (n) \big( 1 - (\log n)^{-\epsilon_{1}} \big) \Big\}, \\ 
&W^{11} = \Big\{ d_{{\cal G}_{n/2, t'_{l-j}}} \big( S(n/2), S (t'_{l-j}) \big) \ge (l-j) a_{n, -b} \psi (n) \big( 1 - (\log n)^{-\epsilon_{1}} \big) \Big\},
 \end{align*}
 where $\epsilon_{1} >0$ in $W^{10}$ and $ W^{11}$ will be given later.

 An argument similar to that in \cite[Lemma 4.2.1]{exact} gives that 
 $$D_{n}^{1} + D_{n}^{2} - D_{n} \ge l a_{n, -b} \psi (n) \big( 1 - (\log n)^{-\epsilon_{1}} \big) -16 a_{n, -b} \psi (n),$$
 on the event $W^{1} \cap W^{2} \cap \cdots \cap W^{11}$. Furthermore, an imitation of \cite[(4.6)]{exact} shows that 
 \begin{align}
 &E (D_{n}^{1} + D_{n}^{2} - D_{n} ) \notag \\
 &\ge \sum_{l=(\log n)^{2 \epsilon_{2}}}^{N/2} \sum_{j= (\log n)^{- \epsilon_{2}} l}^{l- (\log n)^{- \epsilon_{2}} l} \Big\{ l a_{n, -b} \psi (n) \big( 1 - (\log n)^{-\epsilon_{1}} \big) -16 a_{n, -b} \psi (n) \Big\} \notag \\
 & \ \ \ \ \ \ \ \ \ \ \ \ \ \ \ \  \ \ \ \ \ \ \ \ \ \ \ \ \ \  \ \ \ \  \ \ \times P \big( W^{1} \cap W^{2} \cap \cdots \cap W^{11} \big) \notag  \\
 &+ \sum_{l=\frac{N}{2} \big( 1 + (\log n)^{- \epsilon_{2}} \big)}^{N} \sum_{j= l-N/2}^{N/2} \Big\{ l a_{n, -b} \psi (n) \big( 1 - (\log n)^{-\epsilon_{1}} \big) -16 a_{n, -b} \psi (n) \Big\} \notag \\
 &\ \ \ \ \ \ \ \ \ \ \ \ \ \ \ \  \ \ \ \ \ \ \ \ \ \ \ \ \ \  \ \ \ \  \ \ \times P \big( W^{1} \cap W^{2} \cap \cdots \cap W^{11} \big), \label{lower2}
 \end{align}
 where $\epsilon_{2} >0$ will be defined later. Thus, the rest is to prove the following lemma (which corresponds to \cite[Lemma 4.2.2]{exact}):
 
 \begin{lemma}\label{last}
 Suppose that $l$ and $j$ satisfy either 
 \begin{equation}\label{first}
 (\log n)^{2 \epsilon_{2}} \le l \le N/2  \ \ \text{ and } \ \  (\log n)^{- \epsilon_{2}} l \le j \le l- (\log n)^{- \epsilon_{2}} l
 \end{equation}
 or 
 \begin{equation}\label{second}
 \frac{N}{2} \big( 1 + (\log n)^{- \epsilon_{2}} \big) \le l \le N  \ \ \text{ and } \ \ l-N/2 \le j \le N.
 \end{equation}
  Then we have 
 $$P \big( W^{1} \cap W^{2} \cap \cdots \cap W^{11} \big) \ge P (W^{1}) \Big\{ 1 - C (\log n)^{-\frac{1}{12}}  \Big\}$$
 for some universal constant $C>0$. 
 \end{lemma}

 \textit{Proof of Lemma \ref{last}.} 
 Since the case of \eqref{second} is similar to that of \eqref{first}, we will only consider the case of \eqref{first}. It suffices to show that 
 \begin{equation}\label{check}
 P ( W^{1} \cap (W^{i})^{c} ) \le C P (W^{1})  (\log n)^{-\frac{1}{12}},
 \end{equation}
 for each $i=2,3, \cdots , 11$.
 
 Let us begin with $i=2$. Suppose that $W^{1} \cap (W^{2})^{c}$ occurs. This implies that either $S [0, t_{j}] \cap S [n/2, n] \neq \emptyset$ or $S[t_{j-1}, n/2] \cap S[t_{l-j}', n] \neq \emptyset $ occurs. However,  an easy adaptation of the argumentation used to derive \eqref{deco} shows that
\begin{align*}
&P \Big( \{ I_{j} \leftrightarrow I_{l-j}' \} \cap \big[ \{ S[0, t_{j}] \cap S[n/2, n] \neq \emptyset \} \cup \{ S[t_{j-1}, n/2] \cap S[t_{l-j}', n] \neq \emptyset \} \big] \Big)   \\
&\le C (\log n)^{-1} (\log \log n)^{C} P \big( I_{j} \leftrightarrow I_{l-j}'  \big).
\end{align*}
 So, the $i=2$ case satisfies \eqref{check}. Since the complement of $W^{3}$ also implies a similar intersection of $S$, we see that the $i=3$ case also satisfies \eqref{check}.
 
 As for the case of $i=4, 5$, this is completely the same as \eqref{deco}. Thus, we know that \eqref{check} is satisfied for $i=4,5$. 
 
 The case of $i=6,7$ is simpler. Applying Lemma \ref{uppertail} (with $\epsilon = 0$), we have  
 $$P \big( (W^{6})^{c} \big) \vee P \big( (W^{7})^{c} \big) \le C (\log n)^{-\frac{13}{4}}.$$
 On the other hand, by Proposition \ref{fnk}, 
 $$P (W^{1}) \asymp l^{-2} (\log n)^{-1} \ge (\log n)^{-\frac{19}{6}}.$$
 Thus, \eqref{check} is satisfied for $i=6,7$. 
 
 The remaining cases ($i=8,9,10,11$) are all similar. Since the case $i = 11$ is discussed in detail in the proof of \cite[Lemma 4.2.2]{exact}, we will provide the details for the case $i = 8$ here. By   \cite[Lemma 7.7.4]{Lawb}, it follows that if 
 $$W' = \Big\{ \exists U \in [t_{j-2} - a_{n,-6}, t_{j-2}] \text{ such that } U \text{ is a cut time up to } n/2 \Big\},$$
 then $W' \subset W^{8}$ and $P (W') = 1- O \big( (\log n)^{-1} \log \log n\big)$. We aim to show that $(W')^{c}$ and $W^{1}$ are almost independent. By translating and reversing a path, we have 
 $$P \big( W^{1} \cap (W')^{c} \big) = P \big( W'' \cap  W''' \big),$$
 where 
 \begin{align*}
 &W'' = \Big\{ S^{1} [0, a_{n,-b}] \cap S^{2} [(l-2)a_{n,-b}, (l-1)a_{n,-b}] \neq \emptyset \Big\} \ \  \text{ and } \\
 &W''' = \Big\{ \nexists U \in [a_{n,-b}- a_{n,-6}, a_{n,-b}] \text{ such that } \\
 & \ \ \ \ \ \ \ \ \ \ \ \  \ \big( S^{1} [0, (N/2 -j+1) a_{n,-b}] \cup S^{2} [0, U] \big) \cap S^{2} [U+1, (j-1)a_{n,-b}] = \emptyset \Big\}.
 \end{align*}
 Here $S^{1}$ and $S^{2}$ stand for independent simple random walks in $\mathbb{Z}^{4}$ started at the origin.
 On the other hand,   \cite[Lemma 1.5.1]{Lawb} shows that 
\begin{align*}
 &P \Big( \text{diam} \big( S^{1}[0, a_{n,-b}] \big) \vee \text{diam} \big(  S^{2} [(l-2)a_{n,-b}, (l-1)a_{n,-b}] \big) \ge \sqrt{a_{n,-b}} (\log \log n)^{2} \Big) \\
 &\le C (\log n)^{-100}.
 \end{align*}
Here for $A \subset \mathbb{R}^{4}$, we write $\text{diam} (A) = \sup \{ |x-y| \ : \  x, y \in A \}.$
 Thus, let
 \begin{align*}
& W_{\ast} = \Big\{ \text{diam} \big( S^{1}[0, a_{n,-b}] \big) \vee \text{diam} \big(  S^{2} [(l-2)a_{n,-b}, (l-1)a_{n,-b}] \big) \le \sqrt{a_{n,-b}} (\log \log n)^{2} \Big\},  \\ 
&  W_{\ast \ast} = \Big\{ \big| S^{2} \big( (l-2)a_{n,-b} \big) \big| \le 2 \sqrt{a_{n,-b}} (\log \log n)^{2} \Big\}, \\
& \text{ and } \ W_{\ast \ast \ast } = \Big\{ \text{dist} \Big( S^{2} \big( (l-2)a_{n,-b} \big),  S^{1} [0, n] \Big) \ge \sqrt{n} (\log n)^{-20} \Big\}.
\end{align*}
 Note that $W'' \cap W_{\ast} \subset W_{\ast \ast}$. By \cite[Theorem 1.2.1]{Lawb} and \cite[Proposition 1.5.10]{Lawb}, we have
 $$P (W_{\ast \ast \ast}) = 1- O \big( (\log n)^{-10} \big). $$
 
 Therefore, combining these estimates, we have 
 \begin{align*}
 &P \big( W'' \cap W''' \big) \le P \big( W" \cap W''' \cap W_{\ast} \big) + C (\log n)^{-100} \\
 &\le P \big( W'' \cap W''' \cap W_{\ast \ast} \big) + C (\log n)^{-100} \le  P \big( W'' \cap W''' \cap W_{\ast \ast} \cap W_{\ast \ast \ast} \big) + C (\log n)^{-10}.
 \end{align*}

In this case, we apply the `freezing lemma' \eqref{freez} according to the following procedure.

\begin{itemize}
\item We freeze $S^{1}[0,n]$ in the present setting. That is, in the definition of $Y_q$ in \eqref{YQ}, we interchange the roles of $S^{1}$ and $S^{2}$. Consequently, in this case $Y_q$ is an $S^{1}[0,n]$-measurable random variable.

\item We use the Markov property for $S^{2}$ at time $(l-2)a_{n,-b}$. (We note that $W'''$ is an event measurable with respect to $S^{1}$ and $S^{2} [0, (j-1)a_{n,-b}]$. We also recall that  $j-1 < l-2$ in the present setting.) Thanks to the event $W_{\ast\ast\ast}$, we can choose $q = 20$ in \eqref{YQ}.

\item Applying the freezing lemma \eqref{freez} with $p = 10$, we obtain that, uniformly over the paths of $S^{1}[0,n]$ with probability $1 - O\big((\log n)^{-10}\big)$, the probability that $S^{2}$ intersects $S^{1}[0,a_{n,-b}]$ is bounded above by $\frac{C(\log\log n)^C}{\log n}$.
    
\end{itemize}
By applying the freezing lemma \eqref{freez} as above, we have

 \begin{align*}
 &P \big( W'' \cap W''' \cap W_{\ast \ast} \cap W_{\ast \ast \ast} \big) + C (\log n)^{-10} \\
 &\le  P \Big( W_{\ast \ast} \cap W''' \Big) \times \frac{ C (\log \log n)^{C}}{\log n} + C (\log n)^{-10}.
 \end{align*}
 However, using the Markov property for $S^{2}$ at time $(j-1)a_{n,-b}$ (again, note that $W'''$ is an event measurable with respect to $S^{1}$ and $S^{2} [0, (j-1)a_{n,-b}]$), we have 
 \begin{align*}
 &P \Big( W_{\ast \ast} \cap W''' \Big) \\
 &\le P (W''') \max_{x \in \mathbb{Z}^{4}}  P^{x} \Big( \big| S \big( (l-j-1) a_{n,-b} \big) \big| \le 2 \sqrt{a_{n,-b}} (\log \log n)^{2} \Big)  \\
 &\le C (\log n)^{-1} (\log \log n) (l-j)^{-2} (\log \log n)^{8} \\
 &\le C (\log n)^{-1} (\log \log n)^{9} l^{-2} (\log n)^{2 \epsilon_{2}},
 \end{align*}
 where we used 
 \begin{itemize}
 \item   \cite[Theorem 1.2.1]{Lawb} in the second inequality to show that the maximum is smaller than $C (l-j)^{-2} (\log \log n)^{8}$,
 \item \eqref{first} in the last inequality to show that $l-j \ge l (\log n)^{-\epsilon_{2}}$.
 \end{itemize}
 Thus,
 \begin{align*}
 &P \big( W^{1} \cap (W^{8})^{c} \big) \le P \big( W'' \cap W''' \big) \le C (\log n)^{-2 +2 \epsilon_{2}} l^{-2} (\log \log n)^{C} \\
 &\le C  (\log n)^{-1 +2 \epsilon_{2}} (\log \log n)^{C} P (W^{1}).
 \end{align*}
 (Since we must show that
$P \big( W^{1} \cap (W^{8})^{c} \big)$
is of a smaller order than
$P(W^{1}) \asymp l^{-2} (\log n)^{-1}$, 
a somewhat more involved argument is required than that used in the proof of \eqref{deco}, as shown above.) 
 
 Although we have not yet specified the value of $\epsilon_{2}$, the constraint from the inequality above is that $\epsilon_{2} \in (0, 1/2)$. Since $P \big( W^{1} \cap (W^{9})^{c} \big)$ satisfies the same inequality, \eqref{check} is satisfied for $i=8,9$. 
 
 As for $i=10, 11$, using Lemma \ref{var}, we see that 
 $$P \big( (W^{10})^{c} \big) \vee P \big( (W^{11})^{c} \big) \le C (\log n)^{-\frac{1}{4} + 2 \epsilon_{1}}.$$
 The same argument as above, we can prove that $W^{1}$ and $(W^{10})^{c}$ are almost independent in the sense that 
 \begin{align*}
 P \big( W^{1} \cap (W^{10})^{c} \big) &\le C P \big( (W^{10})^{c} \big) l^{-2} (\log n)^{-1 + 2 \epsilon_{2}} (\log \log n)^{C} \\
 &\le C (\log n)^{-\frac{1}{4} + 2 \epsilon_{1}} P \big( W^{1}  \big) (\log n)^{2 \epsilon_{2}} (\log \log n)^{C}.
 \end{align*}
 Choosing 
 \begin{equation}\label{epsilon}
 \epsilon_{1} := \frac{1}{24} \ \ \text{ and } \ \ \epsilon_{2} := \frac{1}{96},
 \end{equation}
 we get \eqref{check} for all $i$. \hspace{122mm} $\square$
 
 We are now ready to finish the proof of Proposition \ref{lower}. Recall that $\epsilon_{1}$ and $\epsilon_{2}$ are chosen as in \eqref{epsilon}. Combining \eqref{lower2} with Proposition \ref{fnk} and  Lemma \ref{last}, a careful calculation shows that
 \begin{align*}
 &E (D_{n}^{1} + D_{n}^{2} - D_{n} ) \notag \\
 &\ge \sum_{l=(\log n)^{2 \epsilon_{2}}}^{N/2} \sum_{j= (\log n)^{- \epsilon_{2}} l}^{l- (\log n)^{- \epsilon_{2}} l} \Big\{ l a_{n, -b} \psi (n) \big( 1 - (\log n)^{-\epsilon_{1}} \big) -16 a_{n, -b} \psi (n) \Big\} \\
 &\ \ \ \ \ \ \ \ \ \ \ \ \ \ \ \ \ \ \ \ \ \ \ \ \ \ \ \ \ \ \times  P \big( W^{1} \cap W^{2} \cap \cdots \cap W^{11} \big) \notag  \\
 &+ \sum_{l=\frac{N}{2} \big( 1 + (\log n)^{- \epsilon_{2}} \big)}^{N} \sum_{j= l-N/2}^{N/2} \Big\{ l a_{n, -b} \psi (n) \big( 1 - (\log n)^{-\epsilon_{1}} \big) -16 a_{n, -b} \psi (n) \Big\} \\
 &\ \ \ \ \ \ \ \ \ \ \ \ \ \ \ \ \ \ \ \ \ \ \ \ \ \ \ \ \ \ \times P \big( W^{1} \cap W^{2} \cap \cdots \cap W^{11} \big) \\
 &\ge \sum_{l=(\log n)^{2 \epsilon_{2}}}^{N/2} \sum_{j= (\log n)^{- \epsilon_{2}} l}^{l- (\log n)^{- \epsilon_{2}} l} \Big\{ l a_{n, -b} \psi (n) \big( 1 - (\log n)^{-\epsilon_{1}} \big) -16 a_{n, -b} \psi (n) \Big\} \\ 
 &\ \ \ \ \ \ \ \ \ \ \ \ \ \ \ \ \ \ \ \ \ \ \ \ \ \ \ \ \ \ \times P (W^{1}) \Big\{ 1 - C (\log n)^{-\frac{1}{12}}  \Big\} \\
 &+ \sum_{l=\frac{N}{2} \big( 1 + (\log n)^{- \epsilon_{2}} \big)}^{N} \sum_{j= l-N/2}^{N/2} \Big\{ l a_{n, -b} \psi (n) \big( 1 - (\log n)^{-\epsilon_{1}} \big) -16 a_{n, -b} \psi (n) \Big\} \\
 &\ \ \ \ \ \ \ \ \ \ \ \ \ \ \ \ \ \ \ \ \ \ \ \ \ \ \ \ \ \ \times P (W^{1}) \Big\{ 1 - C (\log n)^{-\frac{1}{12}}  \Big\} \\
 &\ge \sum_{l=(\log n)^{2 \epsilon_{2}}}^{N/2} \sum_{j= (\log n)^{- \epsilon_{2}} l}^{l- (\log n)^{- \epsilon_{2}} l} \Big\{ l a_{n, -b} \psi (n) \big( 1 - (\log n)^{-\epsilon_{1}} \big) -16 a_{n, -b} \psi (n) \Big\}   \\
 & \ \ \ \ \ \ \ \ \ \ \ \ \ \ \ \ \ \  \ \ \ \ \ \ \ \ \ \ \ \  \ \ \ \ \ \times \frac{1}{2 \log n} \, \log \Big[ 1 + \frac{1}{l^{2} -1} \Big] \,  \Big\{ 1 + O \Big( (\log n)^{-\frac{1}{12}} \Big) \Big\}  \\
 &+ \sum_{l=\frac{N}{2} \big( 1 + (\log n)^{- \epsilon_{2}} \big)}^{N} \sum_{j= l-N/2}^{N/2} \Big\{ l a_{n, -b} \psi (n) \big( 1 - (\log n)^{-\epsilon_{1}} \big) -16 a_{n, -b} \psi (n) \Big\}   \\
 & \ \ \ \ \ \ \ \ \ \ \ \ \ \ \ \ \ \ \ \ \ \ \ \ \  \  \ \ \ \ \ \ \ \ \ \ \ \times \frac{1}{2 \log n} \, \log \Big[ 1 + \frac{1}{l^{2} -1} \Big] \,  \Big\{ 1 + O \Big( (\log n)^{-\frac{1}{12}} \Big) \Big\} \\
 &\ge  \sum_{l=(\log n)^{2 \epsilon_{2}}}^{N/2} \sum_{j= (\log n)^{- \epsilon_{2}} l}^{l- (\log n)^{- \epsilon_{2}} l} l a_{n, -b} \psi (n) \frac{1}{2 \log n} \, \log \Big[ 1 + \frac{1}{l^{2} -1} \Big] \Big\{ 1 - C (\log n)^{-2 \epsilon_{2}}  \Big\} \\
 &+  \sum_{l=\frac{N}{2} \big( 1 + (\log n)^{- \epsilon_{2}} \big)}^{N} \sum_{j= l-N/2}^{N/2} l a_{n, -b} \psi (n) \frac{1}{2 \log n} \, \log \Big[ 1 + \frac{1}{l^{2} -1} \Big] \Big\{ 1 - C (\log n)^{-  \epsilon_{1}}  \Big\} \\
 &\ge \sum_{l=1}^{N/2} \sum_{j= 1}^{l} l a_{n, -b} \psi (n) \frac{1}{2 \log n} \,  \frac{1}{l^{2} } +  \sum_{l=\frac{N}{2}}^{N} \sum_{j= l-N/2}^{N/2} l a_{n, -b} \psi (n) \frac{1}{2 \log n} \,  \frac{1}{l^{2} } \\
 & \ \ \ \ \ - O \Big( n \psi (n) (\log n)^{-1-\epsilon_{2}}\Big)  \\
 &\ge \sum_{l=\frac{N}{2}}^{N} N  a_{n, -b} \psi (n) \frac{1}{2 \log n} \,  \frac{1}{l} - O \Big( n \psi (n) (\log n)^{-1-\epsilon_{2}}\Big). \\
 \end{align*}
 (It is again here that we make essential use of the sharp estimate obtained in Proposition \ref{fnk}.)
However, the last quantity appearing in the above chain of inequalities coincides with
$$ \frac{\log 2}{2 \log n} n \psi (n) \Big\{ 1 - O \Big( (\log n)^{-\epsilon_{2}} \Big) \Big\} = \frac{\log 2}{2 \log n} n \psi (n/2) \Big\{ 1 - O \Big( (\log n)^{-\epsilon_{2}} \Big) \Big\},$$
where we used \eqref{eqiv2} in the last equation. This finishes the proof.
\end{proof}

\begin{proof}[Proof of Theorem \ref{main}]
Combining Lemma \ref{deriv}, \eqref{eqiv2}, Proposition \ref{upper} and Proposition \ref{lower}, we obtain \eqref{main2}. The first assertion \eqref{main1} can be proved similarly. 
\end{proof}

\bibliographystyle{plain} 
\bibliography{4d}

\end{document}